\documentclass[11pt]{amsart}
\usepackage{amssymb}
\usepackage{verbatim}

\numberwithin{equation}{section}

\newtheorem{lemma}[equation]{Lemma}
\newtheorem{thm}[equation]{Theorem}
\newtheorem{conjecture}[equation]{Conjecture}
\newtheorem{cor}[equation]{Corollary}
\newtheorem{prop}[equation]{Proposition}
\newtheorem{fact}[equation]{Fact}
\newtheorem{claim}[equation]{Claim}
\newtheorem*{lemA}{Lemma A}

\theoremstyle{remark}
\newtheorem*{Remarkapp}{Remark}
\newtheorem{remark}[equation]{Remark}
\newtheorem*{notation}{Notation}

\theoremstyle{definition}
\newtheorem*{defi}{Definition}
\newtheorem{example}[equation]{Example}

\renewcommand{\bar}[1]{#1\llap{$\overline{\phantom{\rm#1}}$}}
\DeclareMathOperator{\Gal}{{Gal}}
\DeclareMathOperator{\lcm}{{lcm}}
\DeclareMathOperator{\Mon}{{Mon}}

\newcommand{\C}{{\mathbb C}}
\newcommand{\Z}{{\mathbb Z}}

\newcommand{\F}{{\mathbb F}}

\newcommand{\Aff}{{\mathbb A}}
\newcommand{\tth}{^{\operatorname{th}}}
\newcommand{\cA}{{\mathcal A}}
\newcommand{\cB}{{\mathcal B}}
\newcommand{\cU}{{\mathcal U}}
\newcommand{\cUk}{{\cU}^k}
\newcommand{\cV}{{\mathcal V}}
\newcommand{\cW}{{\mathcal W}}
\newcommand{\cZ}{{\mathcal Z}}

\newcommand{\cJ}{{\mathcal J}}
\newcommand{\LR}{{\mathcal {LR}}}
\newcommand{\LL}{{\mathcal {LL}}}
\newcommand{\RR}{{\mathcal {RR}}}
\newcommand{\RL}{{\mathcal {RL}}}
\newcommand{\iter}[1]{^{\langle #1\rangle}}
\newcommand{\Line}{{\mathbb P}^1}
\newcommand{\gen}[1]{\mathopen{<}#1\mathclose{>}}
\newcommand{\abs}[1]{\lvert#1\rvert}

\newcommand{\col}{\,{:}\,}

\begin{document}



\title[Polynomial decomposition]{On Ritt's polynomial decomposition
theorems}

\author{Michael E. Zieve}
\address{
  Department of Mathematics,
  Rutgers University,
  Piscataway, NJ 08854, USA
}
\email{zieve@math.rutgers.edu}
\urladdr{www.math.rutgers.edu/$\sim$zieve/}

\author{Peter M\"uller}
\address{
  Institut f\"ur Mathematik,
  Universit\"at W\"urzburg,
  Am Hubland,
  D-97074 \\ W\"urzburg, Germany
}
\email{Peter.Mueller@mathematik.uni-wuerzburg.de}
\urladdr{www.mathematik.uni-wuerzburg.de/$\sim$mueller}

\begin{abstract}
  Ritt studied the functional decomposition of a univariate
  complex polynomial $f$ into prime (indecomposable) polynomials,
  $f=u_1\circ u_2\circ\dots\circ u_r$.  His main achievement was a procedure
  for obtaining any decomposition of $f$ from any other by repeatedly applying
  certain transformations.  However, Ritt's results provide no control on the
  number of times one must apply the basic transformations, which makes his
  procedure unsuitable for many theoretical and algorithmic applications.
  We solve this problem by giving a new description of the collection of all
  decompositions of a polynomial.  One
  consequence is as follows: if $f$ has degree $n>1$ but $f$ is not conjugate
  by a linear polynomial to either $X^n$ or $\pm T_n$ (with $T_n$ the Chebychev
  polynomial), and if the composition $a\circ b$ of polynomials $a,b$ is the
  $k\tth$ iterate of $f$ for some $k>\log_2(n+2)$, then either
  $a=f\circ c$ or $b=c\circ f$ for some polynomial $c$.
  This result has been used by Ghioca, Tucker and Zieve to describe the
  polynomials $f,g$ having orbits with infinite intersection; our results
  have also been used by Medevedev and Scanlon to describe the affine curves
  invariant under a coordinatewise polynomial action.
  Ritt also proved that the sequence
  $(\deg(u_1),\dots,\deg(u_r))$ is uniquely determined by $f$, up to
  permutation.  We show that in fact, up to permutation, the sequence of
  permutation groups $(G(u_1),\dots,G(u_r))$ is uniquely determined by $f$,
  where $G(u)=\Gal(u(X)-t,\C(t))$.  This generalizes both Ritt's invariant and
  an invariant discovered by Beardon and Ng, which turns out to be equivalent
  to the subsequence of cyclic groups among the $G(u_i)$.
\end{abstract}

\date{\today}

\maketitle


\section{Introduction}

Around 1920, Fatou, Julia and Ritt made profound investigations of
functional equations.  For instance, each of them wrote at length on
commuting polynomials, namely $f,g\in\C[X]$ such that $f\circ g = g \circ f$.
This is a particular instance of the general functional equation
$F=f_1\circ \dots \circ f_r = g_1 \circ \dots \circ g_s$
with $f_i,g_j\in\C[X]\setminus \C$, which Ritt studied intensively \cite{Ritt}.
Ritt's strategy was to write each nonlinear $f_i$ and $g_j$
as a composition of minimal-degree nonlinear polynomials, thereby obtaining two
expressions of $F$ as a composition of such `prime' polynomials.
This led him to study the extent of nonuniqueness of the `prime factorization'
of a polynomial under the operation of composition.

The `primes' under this operation are the \emph{indecomposable} polynomials,
namely those $u\in\C[X]$ with $\deg(u)>1$ which cannot be written as the
composition of polynomials of strictly lower degrees.
Given $f\in\C[X]$ with $\deg(f)>1$,
a \emph{complete decomposition} of $f$ is a finite sequence $(u_1,\dots,u_r)$
of indecomposable polynomials $u_i\in\C[X]$
such that $f=u_1\circ\dots\circ u_r$.  Clearly such a complete decomposition
always exists if $\deg(f)>1$; however, it need not be unique.

Ritt gave a procedure for obtaining all complete decompositions of $f$ from a
single such decomposition.  Specifically, he showed that any complete
decomposition of $f$ can be obtained from any other via a finite sequence of
steps, each of which involves replacing two adjacent indecomposables by two
others which have the same composition.
He then solved the equation $a\circ b = c\circ d$ in indecomposable
$a,b,c,d\in\C[X]$.  Up to composing with
linears, the only solutions are the trivial $a\circ b = a\circ b$ and the
nontrivial
\begin{align}
  \label{introeqritt}
  X^n\circ X^s h(X^n) &= X^s h(X)^n\circ X^n\\
  \label{introeqcheb}
  T_n\circ T_m &= T_m\circ T_n,
\end{align}
where $h\in\C[X]$ and $n,s,m$ are positive integers.
The polynomial $T_n$ in \eqref{introeqcheb} is the Chebychev
polynomial, whose definition and basic properties are recalled in
Section~\ref{sec:various}.
We may view \eqref{introeqritt} as the least common
generalization of the fact that $X^n\circ X^s = X^s\circ X^n$ and the fact that
the square of an odd polynomial is even.

Ritt's results are analogous to the classical result in knot theory
that any two knot diagrams belonging to the same knot can be obtained from
one another by a sequence of certain basic transformations known as
Reidemeister moves.  Since in general there is no known bound on the number of
Reidemeister moves required, this result does not resolve the problem of
determining whether two knot diagrams belong to the same knot.
However, the result has been useful in
the study of invariants of knots, since any quantity which is unchanged by
Reidemeister moves is necessarily a knot invariant.  Likewise, Ritt's
reults do not yield any bound on the number of Ritt moves required to pass
between two complete decompositions.  On the other hand, Ritt's results can be
used to determine decomposition invariants.  For instance, by inspecting the
solutions of $a\circ b = c\circ d$ in indecomposable $a,b,c,d\in\C[X]$, we
see that the degrees of $a$ and $b$ are the same as those of $c$ and $d$,
although possibly in reversed order.  It follows from Ritt's procedure that
the sequence of degrees of the indecomposables in a complete decomposition
of $f$ is uniquely determined (up to permutation) by $f$.
Beardon and Ng \cite{BN} used the same method to exhibit another invariant:
given any complete decomposition $f=u_1\circ\dots\circ u_r$, they showed that
the sequence $(\#\Gamma_0(u_1),\dots,\#\Gamma_0(u_r))$ is uniquely determined
(up to permutation) by $f$, where $\Gamma_0(u)$ is the set of linear
$\ell\in\C[X]$ such that $u\circ\ell=u$.  Our first result presents a new
invariant which simultaneously generalizes Ritt's degree invariant and the
Beardon--Ng invariant.

\begin{defi}
  For $u\in\C[X]\setminus\C$, the \emph{monodromy group} $\Mon(u)$ is the
  Galois group of $u(X)-t$ over $\C(t)$, viewed as a group of permutations of
  the roots of $u(X)-t$.
\end{defi}

\begin{thm} \label{monRitt1}
  Pick $f\in\C[X]$ with $\deg(f)>1$. Let $(u_1,\dots, u_r)$ and
  $(v_1,\dots, v_s)$ be two complete decompositions of $f$. Then
  $r=s$, and there is a permutation $\chi$ of the set $\{1,2,\dots,r\}$
  such that $\Mon(u_i)$ and $\Mon(v_{\chi(i)})$ are isomorphic
  permutation groups for all $1\le i\le r$.
\end{thm}

We will show that $\#\Gamma_0(u_i)=1$ unless $\Mon(u_i)$ is cyclic, in which
case $\#\Gamma_0(u_i)=\#\Mon(u_i)$; thus, the Beardon--Ng invariant is
equivalent to the subsequence of cyclic groups in the sequence
$(\Mon(u_1),\dots,\Mon(u_r))$.

Ritt generalized his solution of $a\circ b = c\circ d$ in indecomposable
$a,b,c,d$ to give a similar description of all solutions to
this equation which satisfy $\deg(a)=\deg(d)$ and $\gcd(\deg(a),\deg(c))=1$
(but without assuming indecomposability).  This result has been applied to a
variety of topics, for instance:
\begin{itemize}
  \item The classification of all $f,g\in\Z[X]$ for which the Diophantine
    equation $f(X)=g(Y)$ has infinitely many integer solutions \cite{BT};
  \item The classification of $f,g\in\C[X]$ such that $f^{-1}(A)=g^{-1}(B)$
    for some infinite compact sets $A,B\subset\C$ \cite{P1};
  \item The description of $K[f]\cap K[g]$ and $K(f)\cap K(g)$ for arbitrary
    $f,g\in K[X]$, where $K$ is a field of characteristic zero \cite{BWZ};
  \item A proof that, for $f\in\C((X))\setminus\C(X)$, the set of positive
    integers $m$ for which $f(X^m)\in\C(X)[f]$ consists of the powers of a
    single integer \cite{Zannier2}.
\end{itemize}

However, to date there have been no applications of Ritt's procedure
for passing from one complete decomposition to another (except for the
derivation of the invariants mentioned above).  Our main results transform
Ritt's procedure into an applicable form.  We give a new method for describing
all complete decompositions of a polynomial.
Unlike Ritt's procedure, in our procedure one can determine in advance
exactly how many steps one must perform.

Our method is as follows.  We first write the polynomial $f$ as the composition
of polynomials of two types, which we call blocks: either indecomposable
polynomials which cannot
be transformed into $X^n$ or $T_n$ by composing with linears, or (possibly
decomposable) polynomials which can be so transformed.  Then, when possible,
we combine adjacent blocks of the form $\ell_1\circ X^n\circ\ell_2$
(with the $\ell_i$ linear), so long as their composition again has this form;
and we combine Chebychev blocks similarly.  There can be many different
decompositions of $X^n$, since it is the composition (in any order)
of the various $X^p$ where $p$ runs through the prime factors of $n$ counted
with multiplicities; similar remarks apply to $T_n$.  We obtain complete
decompositions of $f$ by inserting all such complete decompositions of each
$X^n$ or $T_n$ block.  These typically comprise all complete decompositions
of $f$.  There are only two ways to obtain further complete decompositions:
first, if an $X^n$ block is adjacent to a $T_m$ block, and if the linears
between $X^n$ and $T_m$ have appropriate composition, then we can move a
degree-$2$ factor from one block to the other (since $X^2$ is the composition
of $T_2$ with linears); however, we will show that after one degree-$2$ factor
has been moved, no further degree-$2$ factors can be moved in the same
direction.  And second, if an $X^n$ block is adjacent
to an indecomposable of a special form, then we can use (\ref{introeqritt})
to effectively move an $X^k$ sub-block to the other side of the indecomposable;
typically this will change the form of the indecomposable, but we will show
that if $k$ is chosen maximally then no further sub-block of the remaining
$X^{n/k}$ can switch sides with the transformed indecomposable.  We will
give a detailed exposition of our procedure in Section~\ref{sec constraints}.

One application of our results is to the decomposition of
iterates of a polynomial.  Here we write $f\iter{e}$ for the
$e\tth$ iterate of $f$, or in other words the $e\tth$ power of $f$
under the operation of composition.  By convention $f\iter{0}=X$, and
for a linear $\ell\in\C[X]$ we write
$\ell\iter{-1}$ for the inverse of $\ell$, which is again
a linear polynomial.

\begin{thm} \label{iterates}
  Pick $f\in\C[X]$ of degree $n>1$, and suppose there is no linear
  $\ell\in\C[X]$ for which $\ell\circ f\circ\ell\iter{-1}$ is either $X^n$
  or $T_n$ or $-T_n$.  If $a,b\in\C[X]$ satisfy
  $a\circ b = f\iter{e}$ for some $e\ge 1$, then
  \[
    a=f\iter{i}\circ \hat a \quad\text{ and }\quad
    b=\hat b\circ f\iter{j} \quad\text{ and }\quad
    \hat a\circ \hat b = f\iter{k}
  \]
  for some $\hat a,\hat b\in\C[X]$ and $i,j,k\ge 0$ with $k\le \log_2(n+2)$.
\end{thm}

This result says that if $e>\log_2(n+2)$ then every
decomposition $a\circ b=f\iter{e}$ can be obtained from some
decomposition of $f\iter{\lfloor\log_2(n+2)\rfloor}$ by composing
on the outside with several copies of $f$.  The bound on $k$ can be
improved to $k\le\lfloor\log_2(n)\rfloor$ if $n\ne 6$, but in
Example~\ref{ex2} we will show that the bound cannot be improved
further if $n=2^m+2$ with $m\ge 3$.  We will prove a refined version of
Theorem~\ref{iterates} in Section~\ref{sec constraints}, as a consequence
of the stronger Theorem~\ref{preciseiterates}.

Theorem~\ref{iterates} is one of the key ingredients in the companion
paper \cite{GTZ2}, in which the following is proved:

\begin{thm}\label{lines}
For $x_0,y_0\in\C$, if $f,g\in\C[X]$ are nonlinear and the orbits
$\{x_0,f(x_0),f(f(x_0)),\dots\}$ and $\{y_0,g(y_0),g(g(y_0)),\dots\}$ have
infinite intersection, then $f$ and $g$ have a common iterate.
\end{thm}

This question can be translated into a decomposition problem as follows.
Supposing for simplicity that $x_0,y_0\in\Z$ and $f,g\in\Z[X]$, the
hypothesis implies that for any $i,j>0$ the equation
$f\iter{i}(X)=g\iter{j}(Y)$ has infinitely many solutions in integers $X,Y$.
By Siegel's theorem, it follows that $f\iter{i}\circ a = g\iter{j}\circ b$
for some nonconstant $a,b\in\C(X)$ which are Laurent polynomials (i.e.,
rational functions whose denominator is a power of $X$).  Since Ritt's results
have been generalized to the setting of Laurent polynomials
\cite{BT,P2,laurent},
this gives information about decompositions of $f\iter{i}$ and $g\iter{j}$,
which leads to the application of Theorem~\ref{iterates}.
In an earlier paper \cite{GTZ}, Theorem~\ref{lines} was proved in case
$\deg(f)=\deg(g)$; in this special case, the polynomial decomposition arguments
simplify dramatically (essentially because of Corollary~\ref{multipledeg}).
However, the full strength of the results of the present paper seems to be
needed to prove Theorem~\ref{lines} in general.

Another application of the results of this paper was found by Medvedev and
Scanlon: combining our results with a model-theoretic result of Chatzidakis and
Hrushovski, they described the subvarieties of $\Aff^n$ preserved by a
coordinatewise polynomial map
$(x_1,\dots,x_n)\mapsto (f_1(x_1),\dots,f_n(x_n))$ with $f_i\in\C[X]$.

Ritt's results are not well understood:
in many treatments the statements of Ritt's results are either false
\cite{BN, vzG, KL} or weaker than the original versions
\cite{Binder,DW,Engstrom,Fried-Ritt,FriedMacRae,LN,Levi,Mueller}.
In light of this, we have included simplified accounts of
Ritt's proofs (in modern language) in the present paper.

Ritt's proofs have two distinct flavors.  His solution of $a\circ b=c\circ d$
uses that the curve $a(X)=c(Y)$ has genus zero; by expressing this genus in
terms of the ramification in the covers $\Line\to\Line$ corresponding to
$b$ and $d$, one obtains a system of equations satisfied by the ramification
indices, and the main work is to solve this system.  See the appendix for a
simplified version of this argument.  Ritt's proof of his
iterative procedure uses Galois theory to translate the problem to a question
about cyclic groups.  We give an account of this in the next
section, and by extending the method we prove
Theorem~\ref{monRitt1} and other results.
In Section~\ref{sec:various} we give various properties of the special
polynomials occurring in \eqref{introeqritt} and \eqref{introeqcheb}.
We prove our main results in Section~\ref{sec constraints}.  Then in the final
section we briefly survey
related topics, including decomposition of rational functions, decomposition
of polynomials over arbitrary fields, decomposition algorithms, and
monodromy groups of indecomposable polynomials.

\vskip.2cm \noindent
\emph{Acknowledgements:} The first author thanks Dragos Ghioca and
Tom Tucker for a stimulating collaboration on the paper \cite{GTZ},
which led to a conjectural version of Theorem~\ref{iterates};
proving this conjecture was the initial motivation for the research presented
in this paper.  The authors thank Avi Wigderson for suggesting the analogy
with knot theory.


\section{Monodromy groups and Ritt's first theorem}
\label{sec: Ritt1}

In this section we present a Galois-theoretic framework which enables us to
translate many questions about polynomial decomposition into questions about
subgroups of cyclic groups.  In particular, we prove Ritt's result
that one can pass from any complete decomposition of $f$ to any other via
finitely many changes of the following form:

\begin{defi}
  If $(u_1,\dots,u_r)$ and $(v_1,\dots,v_r)$ are complete decompositions of a
  polynomial $f\in\C[X]$, then we say they are \emph{Ritt neighbors} if there
  exists $i$ with $1\le i<r$ such that
  \begin{itemize}
    \item $u_j=v_j$ for $j\notin\{i,i+1\}$, and
    \item $u_i \circ u_{i+1} = v_i \circ v_{i+1}$.
  \end{itemize}
\end{defi}

\begin{thm} \label{Ritt1}
  Pick $f\in\C[X]$ with $\deg(f)>1$.  If $\cU$ and $\cV$
  are complete decompositions of $f$, then there is a finite sequence
  $\mathcal S$ of complete decompositions of $f$ such that
  $\cU,\cV\in\mathcal S$ and every pair of
  consecutive decompositions in $\mathcal S$ are Ritt neighbors.
\end{thm}

We use the following notation in this section.

\begin{itemize}
\item $f$ is a nonconstant polynomial in $\C[X]$
\item $S$ is the set of pairs $(a,b)\in\C[X]^2$ such that $a\circ b=f$
\item $t$ is transcendental over $\C$
\item $L$ is the splitting field of $f(X)-t$ over $\C(t)$
\item $x$ is a root of $f(X)-t$ in $L$
\item $G$ is the monodromy group $\Mon(f)=\Gal(L/\C(t))$
\item $H$ is stabilizer of $x$ in $G$, namely $H=\Gal(L/\C(x))$
\end{itemize}


\subsection{General formalism}

We begin by reviewing the Galois-theoretic framework developed
by Ritt \cite{Ritt} for addressing polynomial decomposition problems.
Our presentation is a modernized and simplified version of Ritt's.

\begin{lemma} \label{luroth}
  The map $\rho\colon (a,b)\mapsto \C(b(x))$ is a surjection from $S$ onto the
  set of fields between $\C(x)$ and $\C(t)$.  For $d\in\C[X]$, we have
  $\rho((a,b))=\C(d(x))$ if and only if there is a linear $\ell\in\C[X]$
  such that $d=\ell\circ b$, in which case $f=(a\circ\ell\iter{-1})\circ d$.
  Moreover,
  $[\C(x)\col\C(b(x))]=\deg(b)$ and $[\C(b(x))\col\C(t)]=\deg(a)$.
\end{lemma}

\begin{proof}
  Let $E$ be a field between $\C(x)$ and $\C(t)$.  By L\"uroth's
  theorem, $E=\C(b(x))$ for some $b\in\C(X)$.  Since $E$ is unchanged if we
  replace $b$ by $\ell\circ b$ where $\ell\in\C(X)$ has degree one, we may
  assume $b(\infty)=\infty$.  Since $t=f(x)$ lies in $\C(b(x))$, we have
  $f(X)=a(b(X))$ for some $a\in\C(X)$.  Now
  $X=\infty$ is the unique preimage of $\infty$ under $f$, and
  $b(\infty)=\infty$, so $X=\infty$ is the unique preimage of
  $\infty$ under each of $a(X)$ and $b(X)$.  Thus $a(X)$ and $b(X)$ are
  polynomials, so $\rho$ is surjective.
  
  By Gauss's lemma, $f(X)-t$ is irreducible over $\C(t)$, so
  $[\C(x)\col\C(t)]=\deg(f)$.  This argument implies the final statement of the
  result.  Moreover, for $d\in\C[X]$ and $(a,b)\in S$, we have
  $\C(b(x))=\C(d(x))$ if and only if $d=\ell\circ b$ and $b=\hat\ell\circ d$
  for some $\ell,\hat\ell\in\C(X)$; then $\ell$ and $\hat\ell$ have degree one,
  and since $b$ and $d$ are polynomials it follows that $\ell$ is linear.
\end{proof}

This result enables us to translate questions about decompositions of $f$
into questions about intermediate fields between $\C(x)$ and $\C(t)$.
Here we define a decomposition of $f$ to be a sequence
$(a_1,\dots,a_r)$ where $f=a_1\circ\dots\circ a_r$ and
each $a_i\in\C[X]$ satisfies $\deg(a_i)>1$ (we do not require the $a_i$ to
be indecomposable).  Such a decomposition corresponds to the chain of fields
$C(x)\supset\C(a_r(x))\supset\C(a_{r-1}\circ a_r(x))\supset\dots\supset
 \C(a_1\circ\dots\circ a_r(x))$.
Letting $\theta$ denote this map from decompositions of $f$ to decreasing
chains of fields from $\C(x)$ to $\C(t)$, we now describe the decompositions
which map to the same chain of fields.

\begin{defi} For $f\in\C[X]$, we say two decompositions
  $(a_1,\dots,a_r)$ and $(b_1,\dots,b_s)$ of $f$ are
  \emph{equivalent} if $r=s$ and there are linear
  $\ell_0,\dots,\ell_r\in\C[X]$, with $\ell_0=\ell_r=X$, such that
  $b_i=\ell_{i-1}\circ a_i\circ \ell_i\iter{-1}$ for $1\le i\le r$.
\end{defi}

This is an instance of the category-theoretic notion of
equivalence of two factorizations of an arrow.
Our next result follows from Lemma~\ref{luroth}.

\begin{cor} \label{chains}
  The map $\theta$ induces a bijection between equivalence classes of
  decompositions of $f$ and decreasing chains of fields from $\C(x)$ to
  $\C(t)$.  If the decomposition $(a_1,\dots,a_r)$ corresponds to the
  chain of fields $\C(x)=E_r>E_{r-1}>\dots>E_0=\C(t)$, then
  $[E_i\col E_{i-1}]=\deg(a_i)$ for $1\le i\le r$.
\end{cor}

We have reduced the study of decompositions of $f$ to the study of
decreasing chains of fields from $\C(x)$ to $\C(t)$.  As usual, the latter
is equivalent to the study of increasing chains of groups from $H$ to $G$.
Concretely, the map
$W\mapsto \C(x)^W$ is a bijection from the set
of groups between $H$ and $G$ to the set
of fields between $\C(x)$ and $\C(t)$, and
$\abs{W_1\col W_2}=[\C(x)^{W_2}\col\C(x)^{W_1}]$ for any groups $W_1,W_2$
with $H<W_2<W_1<G$.  Since there are only finitely many groups between
$H$ and $G$, this implies the following.

\begin{cor} \label{fin}
  There are only finitely many equivalence classes of decompositions of $f$.
\end{cor}

We make one further reduction.  Let $I$ be the inertia group at a place of $L$
lying over $t=\infty$, so $I$ is a cyclic subgroup of $G$, and moreover $I$ is
transitive (since $t=\infty$ is totally ramified in $\C(x)/\C(t)$).
Alternately, we could define $I$ to be the Galois group of $f(X)-t$ over
$\C((1/t))$, so $I$ is cyclic because any finite extension of $\C((1/t))$ is
cyclic, and $I$ is transitive because the monic polynomial whose roots are the
reciprocals of the roots of $f(X)-t$ is Eisenstein over $\C[[1/t]]$ and hence
irreducible over $\C((1/t))$.

The following simple lemma reduces the study of decompositions of $f$ to the
study of increasing chains of groups from $1$ to $I$.

\begin{lemma} \label{L:G=HI}
  Let $I$ be a cyclic subgroup of the finite group $G$, and let $H$ be a
  subgroup of $G$ such that $G=HI$ and $H\cap I=1$.  For any
  group $W$ between $H$ and $G$, we have $W=HJ$ where $J=W\cap I$.
  Conversely, for any subgroup $J$ of $I$, the set $HJ$ is a group if and
  only if $HJ=JH$, in which case $\abs{HJ\col H}=\abs{J}$.
\end{lemma}

\begin{cor} \label{L:G=HIcor}
  For groups $W_1$ and $W_2$ between $H$ and $G$, write $J_i:=W_i\cap I$;
  then
  \begin{enumerate}
    \item[(\thethm.1)] $\gen{W_1,W_2}=H J_1J_2$.
    \item[(\thethm.2)] $\abs{\gen{W_1,W_2}\col H}=
           \lcm(\abs{W_1\col H},\abs{W_2\col H})$.
    \item[(\thethm.3)] $\abs{G\col \gen{W_1,W_2}}=
           \gcd(\abs{G\col W_1},\abs{G\col W_2})$.
    \item[(\thethm.4)] $W_1\cap W_2=H(J_1\cap J_2)$.
    \item[(\thethm.5)] $\abs{(W_1\cap W_2)\col H}=
           \gcd(\abs{W_1\col H},\abs{W_2\col H})$.
    \item[(\thethm.6)] If $\abs{W_1}=\abs{W_2}$, then $W_1=W_2$.
    \item[(\thethm.7)] $N_{G}(H)\le N_{G}(W_1)$.
  \end{enumerate}
\end{cor}

\begin{proof}
  We have $HJ_1J_2=J_1HJ_2=J_1J_2H$, so $\gen{W_1,W_2}=HJ_1J_2$, which implies
  (\ref{L:G=HIcor}.1), (\ref{L:G=HIcor}.2) and (\ref{L:G=HIcor}.3).  Since
  $W_1\cap W_2\cap I=J_1\cap J_2$, we obtain (\ref{L:G=HIcor}.4) and
  (\ref{L:G=HIcor}.5).  For (\ref{L:G=HIcor}.6), note that $I$ has at most one
  subgroup of a given order.
  Finally, for $\tau\in N_G(H)$ we have $H\le W_1^\tau\le G$, so
  (\ref{L:G=HIcor}.7) follows from (\ref{L:G=HIcor}.6).
\end{proof}

\begin{remark}
Corollary~\ref{chains} is implicit in \cite[\S 2]{Ritt} and explicit
in \cite[\S 3]{Levi}.  Corollary~\ref{fin} is due to Ritt \cite[p.~55]{Ritt}.
\end{remark}


\subsection{Greatest common divisors and Ritt's first theorem}
\label{subsec:Ritt1}

In this subsection we prove Ritt's result (Theorem~\ref{Ritt1}) describing how
to obtain any complete decomposition of $f$ from any other.  We then deduce
that the sequence of monodromy groups of the indecomposables in a
complete decomposition of $f$ is uniquely determined (up to permutation) by
$f$.  Our first result describes the left and right greatest common divisors of
two decompositions.

\begin{lemma} \label{gcdlem}
  If $a,b,c,d\in\C[X]\setminus\C$ satisfy $a\circ b = c\circ d$,
  then there exist $\hat a,\hat b,\hat c,\hat d,g,h\in\C[X]$ such that
  \begin{itemize}
    \item $g\circ \hat a = a$,\, $g\circ \hat c=c$,\,
      $\deg(g)=\gcd(\deg(a),\deg(c))$;
    \item $\hat b\circ h = b$,\, $\hat d\circ h=d$,\,
      $\deg(h)=\gcd(\deg(b),\deg(d))$; and
    \item $\hat a\circ \hat b = \hat c\circ \hat d$.
  \end{itemize}
\end{lemma}

\begin{proof}
  Let $a,b,c,d\in\C[X]\setminus\C$ satisfy $a\circ b = c\circ d = f$.
  Let $W_1$ and $W_2$ be the subgroups of $G$ fixing $b(x)$ and $d(x)$,
  respectively, so $H\le W_1,W_2\le G$.  Putting $W:=\gen{W_1,W_2}$,
  Corollary~\ref{chains} implies that the chain of groups
  $H\le W_1\cap W_2\le W_1\le W\le G$
  corresponds to the chain of fields
  $\C(x) \ge \C(h(x)) \ge \C(b(x)) \ge \C(\hat a(b(x))) \ge \C(f(x))$
  with $\hat a,h\in\C[X]$, and by Lemma~\ref{luroth} we have $b=\hat b\circ h$
  and $a=g\circ \hat a$ for some $\hat b,g\in\C[X]$.  Likewise, the chain of
  groups $H\le W_1\cap W_2\le W_2\le W\le G$ corresponds to the chain of fields
  $\C(x) \ge \C(h(x)) \ge \C(d(x)) \ge \C(\hat a(b(x)) \ge \C(f(x))$, so
  $d=\hat d\circ h$ and $\hat a\circ b=\hat c\circ d$ with
  $\hat c,\hat d\in\C[X]$, whence $c=g\circ \hat c$ and
  $\hat a\circ \hat b=\hat c\circ \hat d$.  Finally, the
  statements about degrees follow from (\ref{L:G=HIcor}.3) and
  (\ref{L:G=HIcor}.5).
\end{proof}

\begin{cor} \label{multipledeg}
  Suppose $a,b,c,d\in\C[X]\setminus\C$ satisfy $a\circ b = c\circ d$.
  \begin{enumerate}
    \item[(\thethm.1)] If $\deg(c)\mid\deg(a)$, then $a=c\circ \hat a$ for some
          $\hat a\in\C[X]$.
    \item[(\thethm.2)] If $\deg(d)\mid\deg(b)$, then $b=\hat b\circ d$ for some
          $\hat b\in\C[X]$.
    \item[(\thethm.3)] If $\deg(a)=\deg(c)$, then there is a linear
          $\ell\in\C[X]$ such that $a=c\circ\ell$ and $b=\ell\iter{-1}\circ d$.
  \end{enumerate}
\end{cor}

Assertion (\ref{multipledeg}.3) implies that, up to the insertion of
linears and their inverses between consecutive indecomposables, a complete
decomposition is uniquely determined by the sequence of degrees of the involved
indecomposables.  This yields a refinement of Corollary~\ref{fin}.

We now prove Theorem~\ref{Ritt1}.

\begin{proof}[Proof of Theorem~\ref{Ritt1}]
  By Corollary~\ref{chains} and Lemma~\ref{L:G=HI}, the result is a consequence
  of the following lemma about chains of subgroups of $I$.
\end{proof}

\begin{lemma} \label{Ritt1gp}
  Let $\cJ$ be a set of subgroups of a finite cyclic group $I$, and assume
  that $1,I\in\cJ$ and $\cJ$ is closed under intersections and products.  Let
  $1=A_0<A_1<\dots<A_r=I$ and $1=B_0<B_1<\dots<B_s=I$ be two maximal
  increasing chains of groups in $\cJ$.  Then one can pass from the first chain
  to the second via finitely many steps, each of which involves replacing a
  chain $1=C_0<C_1<\dots<C_r=I$ by a chain $1=D_0<D_1<\dots<D_r=I$ where
  $D_i=C_i$ for all $i$ not equal to a single value $j$ (with $0<j<r$).
\end{lemma}

\begin{proof}
  We proceed by induction on $\abs{I}$.  So suppose the result holds for any
  cyclic group of order less than $\abs{I}$.  Let $\cA=(A_0,A_1,\dots,A_r)$
  and $\cB=(B_0,\dots,B_s)$ be maximal chains as prescribed.  By the inductive
  hypothesis, the conclusion holds for any two chains containing $A_{r-1}$.  So
  suppose $A_{r-1}\ne B_{s-1}$; maximality of the chains implies
  $A_{r-1}B_{s-1}=I$, so there is no group in $\cJ$ properly between
  $A_{r-1}\cap B_{s-1}$ and $A_{r-1}$ (since $A_{r-1}\cap B_{s-1}<J<A_{r-1}$
  implies $\abs{JB_{s-1}\col B_{s-1}}=\abs{J\col J\cap B_{s-1}}=
  \abs{J\col A_{r-1}\cap B_{s-1}}$).  Let
  $1=U_0<U_1<\dots<U_k=A_{r-1}\cap B_{s-1}$ be a maximal increasing chain of
  groups in $\cJ$ contained in $A_{r-1}\cap B_{s-1}$; then
  $\cU:=(U_0,U_1,\dots,U_k,A_{r-1},I)$ is a maximal chain in $\cJ$.  By
  inductive hypothesis, we can pass from $\cA$ to $\cU$ by steps of the
  required type.  In one more such step we replace $\cU$ by
  $\cV:=(U_0,U_1,\dots,U_k,B_{s-1},I)$.  Finally, by inductive hypothesis we
  can pass from $\cV$ to $\cB$ by steps of the required type, and the result
  follows.
\end{proof}

In light of Theorem~\ref{Ritt1}, invariants of pairs of Ritt neighboring
complete decompositions of $f$ are invariants of any pair of complete
decompositions of $f$.
Lemma~\ref{gcdlem} implies the following result about the degrees of the
indecomposables in a pair of Ritt neighbors.

\begin{cor} \label{Ritt move lemma}
  Suppose indecomposable $a,b,c,d\in\C[X]$ satisfy $a\circ b=c\circ d$.
  Then either there is a linear $\ell$ such that $a=c\circ\ell$ and
  $b=\ell\iter{-1}\circ d$, or $\gcd(\deg(a),\deg(c))=\gcd(\deg(b),\deg(d))=1$
  (in which case $\deg(a)=\deg(d)$ and\/ $\deg(b)=\deg(c)$).
\end{cor}

In combination with Theorem~\ref{Ritt1}, this result shows that the sequence
of degrees of the indecomposables in a complete decomposition of $f$ is
uniquely determined (up to permutation) by $f$:

\begin{cor} \label{degrees}
  Pick $f\in\C[X]$ with $\deg(f)>1$. Let $(u_1,\dots, u_r)$ and
  $(v_1,\dots, v_s)$ be complete decompositions of $f$. Then
  $r=s$, and there is a permutation $\chi$ of the set $\{1,2,\dots,r\}$
  such that $u_i$ and $v_{\chi(i)}$ have the same degree for all $1\le i\le r$.
\end{cor}

We now show that $\chi$ can be chosen so that $u_i$ and $v_{\chi(i)}$ share
a finer invariant than the degree: we can require them
to have the same monodromy group.  Note that the monodromy group is a
permutation group whose degree equals the degree of the polynomial, so this
result refines Corollary~\ref{degrees}.

\begin{defi}
  Let $G$ and $\tilde G$ be permutation groups acting on sets $\Omega$ and
  $\tilde\Omega$, respectively. We say that $G$ and $\tilde G$ are isomorphic
  as permutation groups if there is a group isomorphism
  $\phi\colon G\to\tilde G$ and a bijection $\psi\colon\Omega\to\tilde\Omega$
  such that $\psi(\omega^{\tau})=\psi(\omega)^{\phi(\tau)}$ for each
  $\omega\in \Omega$ and $\tau\in G$.
\end{defi}

\begin{thm} \label{monthm}
  Suppose $a,b,c,d\in \C[X]\setminus\C$ satisfy
  $a\circ b = c\circ d$ and
  $\gcd(\deg(a),\deg(c))=1=\gcd(\deg(b),\deg(d))$.  Then $\Mon(a)$ and
  $\Mon(d)$ are isomorphic permutation groups, and so are $\Mon(b)$ and
  $\Mon(c)$.
\end{thm}

\begin{proof}
  Let $x$ be transcendental over $\C$, let $t=a(b(x))$, and let $L$ be a normal
  closure of $\C(x)/\C(t)$. Set $G=\Gal(L/\C(t))$.  Let $U$, $V$, and $H$ be
  the stabilizers in $G$ of $b(x)$, $d(x)$, and $x$, respectively.
  
  Let $N:=\bigcap_{\tau\in G}U^{\tau}$ be the core of $U$ in $G$; then $N$ is
  the kernel of the action of $G$ on the set $G/U$ of right cosets of $U$ in
  $G$.  Thus $\Mon(a)$ is isomorphic to $G/N$ with respect to this action.  Let
  $C:=\bigcap_{v\in V} H^v$ be the core of $H$ in $V$; then $V/C$, in its
  action on the coset space $V/H$, is isomorphic to $\Mon(d)$.
  
  Recall that $G=HI$ with $I$ cyclic.  Since $\abs{U\col H}=\deg(b)$ is coprime
  to $\abs{V\col H}=\deg(d)$, we have $U\cap V=H$.  Then
  $\abs{G\col U}=\deg(a)=\deg(d)=\abs{V\col H}$ implies $G=UV$.  Since $G=UI$,
  we have $N=\bigcap_{\tau\in G}U^{\tau}=\bigcap_{\tau\in I}U^{\tau}\ge
  \bigcap_{\tau\in I}(U\cap I)^{\tau}=U\cap I$. From $U\cap I\le N$ we get
  $U=H(U\cap I)\le HN$, whence $U=HN$ and $VN=VHN=VU=G$.
  
  Since $H=U\cap V$ and $G=UV$, we have
  \[
    C=\bigcap_{v\in V}H^v=\bigcap_{v\in V}(U\cap V)^v
     =(\bigcap_{v\in V}U^v)\cap V=(\bigcap_{v\in UV}U^v)\cap V=N\cap V.
  \]
  Hence the natural map $V\to G/N$ is surjective with kernel $N\cap V=C$, and
  thus induces a natural isomorphism $V/C\to G/N$.  This isomorphism maps
  $H/C$ to $HN/N=U/N$, so $V/C$ and $G/N$ are isomorphic permutation groups
  with respect to their actions on the coset spaces $V/H$ and $G/U$,
  respectively.  Thus $\Mon(d)$ and $\Mon(a)$ are isomorphic as permutation
  groups.
  
  The isomorphy of $\Mon(b)$ and $\Mon(c)$ follows by symmetry.
\end{proof}

Theorem~\ref{monRitt1} follows from Theorem~\ref{Ritt1} and the previous
result.

\begin{remark}
  Letting $a:=X^i h(X)^n$ and $d:=X^i h(X^n)$ with $\gcd(i,n)=1$, we have
  $a\circ X^n = X^n \circ d$, so Theorem~\ref{Ritt1} implies
  that $a$ is indecomposable if and only if $d$ is indecomposable.
  This has been observed previously, and has been regarded as mysterious
  (cf.\ \cite[p.~140]{LN} or \cite[p.~128]{BN}).
  It is explained by Theorem~\ref{monthm}, since a polynomial is indecomposable
  precisely when its monodromy group is primitive.
\end{remark}

\begin{remark} \label{BNrem}
  Beardon and Ng \cite{BN} studied the set $\Gamma_0(u)$ of Euclidean
  isometries of a polynomial $u\in\C[X]\setminus\C$, defined as the set of
  linear $\ell\in\C[X]$ for which $u\circ\ell=f$.  Writing
  $\gamma(u):=\abs{\Gamma_0(u)}$, they showed that if $(u_1,\dots,u_r)$ is a
  complete decomposition of $f$ then $(\gamma(u_1),\dots,\gamma(u_r))$
  is uniquely determined (up to permutation) by $f$.  We now deduce this
  from Theorem~\ref{monRitt1}.  Each element of $\Gamma_0(f)$
  is an automorphism of $\C(x)$ which fixes $\C(f(x))$; conversely, any such
  automorphism is a degree-one rational function fixing the unique preimage
  $X=\infty$ of $f=\infty$, and so lies in $\Gamma_0(f)$.  Thus
  $\Gamma_0(f)\cong N_G(H)/H\cong N_G(H)\cap I$ is cyclic.
  If $f$ is indecomposable
  and $\Gamma_0(f)\ne\{X\}$ then $N_G(H)=G$; since $H$
  contains no nontrivial normal subgroups of $G$ (because $L$ is the normal
  closure of $L^H/L^G$), we must have $H=1$ so $G$ is cyclic of order
  $\gamma(f)$.  Thus the Beardon--Ng invariant amounts to the
  subsequence of cyclic groups among the $\Mon(u_i)$.  Moreover, it is easy to
  see (cf.\ Lemma~\ref{cycdi}) that $G$ is cyclic of order $n$ precisely when
  $f=\ell_1\circ X^n \circ \ell_2$ with $\ell_1,\ell_2$ linear;
  this yields all but one of the new results in \cite{BN}.  The remaining
  result is \cite[Thm.~1.2]{BN}, which says
  $\gamma(a\circ b)\mid\gamma(a)\gamma(b)$; the above interpretation
  (and Corollary~\ref{L:G=HIcor}) implies the refinements
  $\gcd(\deg(b),\gamma(a\circ b))=\gamma(b)$ and
  $\lcm(\deg(b),\gamma(a\circ b))\mid \gamma(a)\deg(b)$.
\end{remark}

\begin{remark}
The crux of Lemma~\ref{gcdlem} is implicit in \cite[pp.~59--60]{Ritt};
a preliminary explicit version is
\cite[Thm.~2.2 and Thm.~3.1]{Engstrom}.
Our version first appeared in \cite[p.~334]{Tortrat}.
Assertion (\ref{multipledeg}.3) is due to Ritt \cite[p.~56]{Ritt};
subsequently, Levi \cite[\S 2]{Levi} proved it by comparing coefficients
(an argument also anticipated by Ritt \cite[p.~221]{Ritt2}), and this 
proof later led to fast decomposition algorithms \cite{vzG,KL}.
Corollary~\ref{Ritt move lemma} occurs in \cite[p.~57]{Ritt}.
In case $a,b,c,d$ are indecomposable, Theorem~\ref{monthm} is shown in
the proof of \cite[Thm.~R.2]{Mueller}.
\end{remark}


\subsection{Ritt's second theorem and Ritt moves}

Ritt's second theorem determines all Ritt neighbors, by solving the equation
$a\circ b = c \circ d$ in indecomposable $a,b,c,d\in\C[X]$.
This equation has the trivial solution
$a=c\circ\ell$ and $b=\ell\iter{-1}\circ d$ with $\ell\in\C[X]$ linear;
by Corollary~\ref{Ritt move lemma}, any other solution satisfies
$\gcd(\deg(a),\deg(c))=\gcd(\deg(b),\deg(d))=1$.
Ritt solved the functional equation assuming only this constraint on
the degrees (and not assuming indecomposability):

\begin{thm}[Ritt] \label{Ritt2}
  Suppose $a,b,c,d\in\C[X]\setminus\C$ satisfy $a\circ b = c\circ d$ and
  $\gcd(\deg(a),\deg(c))=\gcd(\deg(b),\deg(d))=1$.  Then there are linear
  $\ell_j\in\C[X]$ such that (after perhaps switching $(a,b)$ and $(c,d)$)
  the quadruple $(\ell_1\circ a\circ\ell_2,\,
  \ell_2\iter{-1}\circ b\circ\ell_3,\,
  \ell_1\circ c\circ\ell_4,\, \ell_4\iter{-1}\circ d\circ\ell_3)$
  has one of the forms
  \begin{gather}
    \tag{\thethm.1} (T_n,\, T_m,\, T_m,\, T_n)\,\quad\text{or} \\
    \tag{\thethm.2} (X^n\!,\, X^s h(X^n),\, X^s h(X)^n\!,\, X^n),
  \end{gather}
 where $m,n>0$ are coprime, $s\ge 0$ is coprime to $n$, and
 $h\in\C[X]\setminus X\C[X]$.
\end{thm}

We will prove Theorem~\ref{Ritt2} in the appendix.

For applications, it is often useful to combine Theorem~\ref{Ritt2} with
Lemma~\ref{gcdlem} in the following manner:

\begin{cor}
  For $a,b,c,d\in\C[X]\setminus\C$ with $\deg(a)\le\deg(c)$, we have
  $a\circ b = c\circ d$ if and only if there exist
  $\hat a,\hat b,\hat c,\hat d,g,h,\ell_1,\ell_2\in\C[X]$ such that $\ell_i$
  is linear and the following three conditions hold:
  \begin{itemize}
    \item $g\circ \hat a = a$,\, $g\circ \hat c=c$,\,
      $\deg(g)=\gcd(\deg(a),\deg(c))$;
    \item $\hat b\circ h = b$,\, $\hat d\circ h=d$,\,
      $\deg(h)=\gcd(\deg(b),\deg(d))$; and
    \item the tuple
      $(\hat a\circ\ell_1,\,\ell_1\iter{-1}\circ \hat b,\,\hat c\circ\ell_2,\,
        \ell_2\iter{-1}\circ \hat d)$
      has the form of either \emph{(\ref{Ritt2}.1)} or \emph{(\ref{Ritt2}.2)}.
  \end{itemize}
\end{cor}

As noted above, if indecomposable $a,b,c,d\in\C[X]$ satisfy
$a\circ b=c\circ d$, and if there is no linear $\ell\in\C[X]$ for which
$(a,b)=(c\circ\ell,\, \ell\iter{-1}\circ d)$, then $a,b,c,d$ satisfy the
hypotheses of Theorem~\ref{Ritt2}.  In this situation, we
refer to the replacement of $a\circ b$ by $c\circ d$ as a \emph{Ritt move}.
Thus, Theorem~\ref{Ritt1} says that one can pass from any complete
decomposition to any other by a sequence of steps, each of which is either a
Ritt move or is the replacement of consecutive indecomposables $a$ and $b$ by
$a\circ\ell$ and $\ell\iter{-1}\circ b$ for some linear $\ell\in\C[X]$.  Note
that the insertion of $\ell$ and $\ell\iter{-1}$ does not affect the sequence
of degrees of the indecomposables in a complete decomposition, and in a Ritt
move two consecutive coprime degrees in this sequence are interchanged.
Recall that a complete decomposition is uniquely determined by the sequence of
degrees of the involved indecomposables, up to the insertion of pairs of
inverse linears between adjacent indecomposables.
Now pick $f\in\C[X]$ with $\deg(f)>1$, and let
$(u_1,\dots,u_r)$ and $(v_1,\dots, v_r)$ be two complete decompositions of $f$.
Then the sequence $(\deg(v_1),\dots,\deg(v_r))$ can be obtained from the
sequence $(\deg(u_1),\dots,\deg(u_r))$ via finitely many steps, each of
which involves interchanging two consecutive coprime entries.
We note that there are examples in which every permutation of
$(\deg(u_1),\dots,\deg(u_r))$ occurs -- namely, if $f$ is $X^n$ or $T_n$.
However, it turns out that such examples are quite special, and in general
there are further constraints on which permutations can occur.  We will
deduce these constraints in Section~\ref{sec constraints}; naturally, they
depend on the form of the polynomials $u_i$ rather than merely their degrees.


\section{The polynomials involved in Ritt moves}
\label{sec:various}

The difficulty in applying Ritt's results is that, after applying a
Ritt move to an adjacent pair of indecomposables in a complete decomposition,
it may happen that one of the resulting indecomposables
can be involved in another Ritt move, and so on.  In this section we
prove various results about the special polynomials
involved in Ritt moves, which will allow us to control all subsequent
Ritt moves involving the resulting polynomials.  We also
give useful characterizations of these special polynomials.

We will use the following terminology.
\begin{defi}
  We say $f,g\in\C[X]$ are \emph{equivalent} if there are linear
  $\ell_1,\ell_2\in\C[X]$ such that $f=\ell_1\circ g\circ\ell_2$.
\end{defi}

\begin{defi}
  For $f\in\C[X]$, we say $f$ is \emph{cyclic} if it is equivalent to $X^n$ for
  some $n>1$, and we say $f$ is \emph{dihedral} if it is equivalent to $T_n$
  for some $n>2$.
\end{defi}

Here the (normalized) Chebychev polynomial $T_n$ is defined by the
functional equation $T_n(Y+1/Y)=Y^n+1/Y^n$; the classical Chebychev polynomial
$C_n(X)$ defined by $C_n(\theta)=\cos(n\arccos \theta)$ satisfies
$T_n(2X)=2C_n(X)$.
Thus $T_0=2$ and $T_1=X$, and in general
$T_n=XT_{n-1}-T_{n-2}$,
so $T_n$ is a degree-$n$ polynomial and for $n>1$ the two
highest-degree terms of $T_n$ are $X^n$ and
$-nX^{n-2}$.  Also, $T_n\circ (-X)=(-1)^n T_n$ and $T_n\circ T_m=T_m\circ T_n$.


\subsection{Ramification}

We will need some properties of the ramification in the cover
$\pi_f\colon\Line_x\to\Line_t$ corresponding to $f\in\C[X]$,
where $x$ is transcendental over $\C$ and $t=f(x)$ (and $\Line_x$
denotes the projective line with coordinate $x$).  We use the standard
notions of ramification indices, ramification points, and branch points
for the cover $\pi_f$.  We also refer to a point of $\Line_x$ as a
`special point' if it is unramified in $\pi_f$ but its image is a branch
point.  In our concrete setting these notions have the following explicit
definitions:

\begin{defi}
  Pick $f\in\C[X]\setminus\C$.  For $x_0\in\C$, the \emph{ramification index}
  of $x_0$ in $f$, denoted $e_f(x_0)$, is the multiplicity of $x_0$ as a root
  of $f(X)-f(x_0)$.  The \emph{finite ramification points} of $f$ are the
  values $x_0\in\C$ for which $e_f(x_0)>1$.  The \emph{finite branch points} of
  $f$ are the values $f(x_0)$, where $x_0$ is a finite ramification point.  The
  \emph{special points} of $f$ are the values $x_0\in\C$ which are not finite
  ramification points, but for which $f(x_0)$ is a finite branch point.
\end{defi}

We briefly record some standard ramification facts in polynomial language.
If $e_1,e_2,\dots,e_k$ are the multiplicities of the roots of $f(X)-x_0$,
then $\Mon(f)$ contains an element having cycle lengths $e_1,\dots,e_k$
(but this fact is not used in this paper).
Ramification indices are multiplicative in towers: for $f,g\in\C[X]\setminus\C$
and $x_0\in\C$, we have $e_{f\circ g}(x_0) = e_f(g(x_0))\cdot e_g(x_0)$.  The
Riemann--Hurwitz formula for $\pi_f$ says
\[
\deg(f) - 1 = \sum_{x_0\in\C} (e_f(x_0)-1);
\]
since the finite ramification points of $f$ are precisely the roots of the
derivative $f'(X)$, this amounts to writing the degree of $f'(X)$ as
the sum of the multiplicities of its roots.  We will also use the
Riemann--Hurwitz formula for the Galois closure of the cover $\pi_f$, as well
as the following variant of Abhyankar's lemma:

\begin{lemma} \label{abhy}
  Let $F_1,F_2$ be finite extensions of $\C(x)$ whose compositum is $E$.
  Let $Q$ be a place of $F:=F_1\cap F_2$, let $P_i$ be a place of $F_i$
  lying over $Q$, and let $e_i$ denote the ramification index of $P_i/Q$.
  Then for each place $P$ of $E$ lying over both $P_1$ and $P_2$, the
  ramification index of $P/Q$ is $\lcm(e_1,e_2)$.  If $[F_1\col F]$ and
  $[F_2\col F]$ are coprime, then there are precisely $\gcd(e_1,e_2)$ such
  places $P$.
\end{lemma}

\begin{proof}
  Let $L$ be the Galois closure of $E/F$, let $G:=\Gal(L/F)$, and let
  $H_1$, $H_2$, and $H$ be the stabilizers in $G$ of $F_1$, $F_2$, and $E$,
  respectively.  Let $R$ be a place of $L$ lying over $P_1$ and $P_2$, let
  $I$ be the inertia group of $R/Q$, and let $P$ be the place of $E$ lying
  under $R$.  Then the inertia groups of $R/P$ and $R/P_i$ are $I\cap H$
  and $I\cap I_i$; since $H=H_1\cap H_2$, it follows that
  $I\cap H=(I\cap H_1)\cap (I\cap H_2)$.  Cyclicity of $I$ implies
  $\abs{I\cap H}=\gcd(I\cap H_1,I\cap H_2)$, so the ramification index
  of $P/Q$ is $\abs{I\col I\cap H}=\lcm(\abs{I\col I\cap H_1},
  \abs{I\col I\cap H_2})=\lcm(e_1,e_2)$.

  Since $G$ (resp.~ $H_i$) acts transitively on the places of $R$ lying over
  $Q$ (resp.~ $P_i$), and $I$ is the stabilizer of $R$, for $g\in G$ the place
  $gR$ lies over $P_i$ if and only if $g\in H_iI$.  Thus the places of $E$
  lying over $P_1$ and $P_2$ are the restrictions to $E$ of places $gR$ with
  $g\in H_1I\cap H_2I$; since the restrictions to $E$ of $g_1R$ and $g_2R$ are
  the same precisely when $g_2^{-1}g_1\in HI$, the number of places of $E$
  lying over $P_1$ and $P_2$ is $\abs{H_1I\cap H_2I}/\abs{HI}$.  Note that
  $\abs{HI}=\abs{H}\abs{I\col I\cap H}=\abs{H}\lcm(e_1,e_2)$, so we must show
  that $\abs{H_1I\cap H_2I}=\abs{H}e_1e_2$.  Assume $[F_1\col F]$ and
  $[F_2\col F]$ are coprime, or equivalently $\abs{G\col H_1}$ and
  $\abs{G\col H_2}$ are coprime.  Then $\abs{G\col H_1}$ divides
  $\abs{G\col H}=\abs{G\col H_2}\abs{H_2\col H}$, and so divides
  $\abs{H_2\col H}$, so $\abs{G}\le\abs{H_1}\abs{H_2}/\abs{H}=\abs{H_1H_2}$,
  whence $G=H_1H_2$.  Thus the set of right-cosets
  $H\backslash G$ has the same cardinality as
  $H_1\backslash G\times H_2\backslash G$; since $gH\mapsto (gH_1,gH_2)$
  defines an injection
  $\rho\colon H\backslash G\to H_1\backslash G\times H_2\backslash G$, it
  follows that $\rho$ is bijective.  
  Finally, $e_i=\abs{I\col I\cap H_i} = \abs{H_iI}/\abs{H_i}$, so
  $e_1e_2=\abs{\rho^{-1}(H_1I,H_2I)}$, whence indeed
  $e_1e_2\abs{H}=\abs{H_1I\cap H_2I}$ as desired.
\end{proof}

We now characterize cyclic and dihedral polynomials in terms of their
ramification.

\begin{lemma} \label{ramchar}
  Pick $f\in\C[X]$ with $\deg(f)>1$.  Then $f$ is cyclic if and only if $f$
  has a unique finite branch point (or equivalently, $f$ has a unique finite
  ramification point).  Likewise, $f$ is dihedral if and only if $f$ has
  precisely two finite branch points and every finite ramification point has
  ramification index $2$; these conditions imply there are precisely two
  special points.
\end{lemma}

\begin{proof}
  If $f$ has a unique finite branch point or a unique finite ramification
  point, then by Riemann--Hurwitz it has both a unique finite branch point
  $\alpha$ and a unique finite ramification point $\beta$.  Thus
  $f(X+\beta)-\alpha$ has no nonzero roots, and so equals $\gamma X^{\deg(f)}$.

  Now suppose that $f$ has precisely two finite branch
  points, and further that every finite ramification point has ramification
  index $2$.  Letting $L$ denote the Galois closure of the extension
  $\C(x)/\C(f(x))$, Lemma~\ref{abhy} implies that $L/\C(f(x))$ is ramified
  over precisely two finite places of $\C(f(x))$ (both with ramification index
  $2$) and over the infinite place (with ramification index $n$).  By
  Riemann--Hurwitz we compute $[L\col\C(x)]=2$, so Lemma~\ref{quadext}
  implies $f$ is dihedral.

  Finally, if $f$ is cyclic or dihedral then it is well-known (and easy to
  verify) that the ramification of $\pi_f$ is as described.
\end{proof}

\begin{lemma} \label{quadext}
  Pick $f\in \C[X]$ with $\deg(f)>1$, let $x$ be transcendental over\/ $\C$,
  and let $L$ be the Galois closure of\/ $\C(x)/\C(f(x))$.  Then $L=\C(x)$
  if and only if $f$ is cyclic, and $[L\col\C(x)]=2$ if and only if $f$ is
  dihedral.
\end{lemma}

\begin{proof}
  If $L=\C(x)$ then all points of $L$ lying over the same point of $\C(f(x))$
  are in a single orbit of $\Gal(L/\C(f(x)))$, and so have the same
  ramification index.  By Riemann--Hurwitz, it follows that $f$ has a unique
  finite ramification point, so $f$ is cyclic.  Conversely, if $f$ is cyclic
  then visibly $\C(x)/\C(f(x))$ is Galois, and likewise if $f$ is dihedral
  then $[L\col\C(f(x))]=2$.

  Henceforth assume $[L\col\C(f(x))]=2$.  Then each root of $f(X)-f(x)$ has
  degree at most $2$ over $\C(x)$,
  so $f(X)-f(x)$ is the product of irreducibles $\Phi_i(X,x)\in \C[X,x]$ each
  of which has $X$-degree at most $2$.  By symmetry, also the $x$-degree of
  each $\Phi_i(X,x)$ is at most $2$.  The leading coefficient of $\Phi_i(X,x)$
  (viewed as a polynomial in $X$ with coefficients in $\C[x]$) divides the
  corresponding leading coefficient of $f(X)-f(x)$, and hence lies in $\C^*$;
  likewise the same property holds if we interchange $x$ and $X$, so
  $\Phi_i(X,x)$ has total degree at most $2$.  Since $L\ne\C(x)$, some $\Phi_i$
  has $X$-degree $2$.  Let $x_i\in L$ satisfy $\Phi_i(x_i,x)=0$, so
  $L=\C(x,x_i)$.  Since $\Phi_i$ has total degree $2$, the genus of $L$ is
  zero, so $L=\C(z)$ for some $z$.

  Set $x=b(z)$ with $b\in \C(X)$ of degree $[\C(z):\C(x)]=2$. The
  infinite place of $\C(t)$ is totally ramified in each conjugate of
  $\C(x)$, so by Lemma~\ref{abhy} the infinite place of $\C(x)$ is
  unramified in $\C(z)$. Thus, after a linear fractional change of $z$ and
  a linear change of $x$ (and $f$), we have $x=z+1/z$.
  
  For each $\C$-automorphism $\sigma$ of $\C(z)$, the image $z^{\sigma}$ of $z$
  is a linear fractional change of $z$.  If $\sigma$ fixes $t:=f(x)$, then
  $\sigma$ fixes the set of values of $z$ which map to $t=\infty$, namely
  $\{0,\infty\}$, so $z^{\sigma} = \alpha z^{\epsilon}$ with $\alpha\in\C^*$
  and $\epsilon\in\{1,-1\}$.  Let $\tau$ generate $\Gal(\C(z)/\C(x))$,
  so $z^\tau=1/z$. Then
  $\Gal(\C(z)/\C(t))=C\gen{\tau}$, where $C$ consists of the
  maps $z\mapsto\alpha z$ with $\alpha^n=1$.  The fixed field of $C$ is
  $\C(z^n)$, and the fixed field of $C\gen{\tau}$ is
  $\C(t)=\C(z^n+1/z^n)$. Thus $t=\ell(z^n+1/z^n)$ for some degree-one
  rational function $\ell$. Since $\ell(\infty)=\infty$, in fact
  $\ell$ is a polynomial, so a linear change to $t$ makes
  $t=z^n+1/z^n$. But $f(z+1/z)=z^n+1/z^n$ implies $f=T_n$, and the
  result follows.
\end{proof}

For $n>1$, the unique finite branch point of $X^n$ is $0$, which is also the
unique finite ramification point.  For $n>2$, the special points of $T_n$ are
$2$ and $-2$, which are also the finite branch points of $T_n$.

The analogous ramification characterization of $X^s h(X)^n$ is immediate:

\begin{lemma} \label{twistramchar}
  Pick $f\in\C[X]\setminus\C$ and integers $n>0$ and $s\ge 0$.
  Then $f$ is equivalent to $X^s h(X)^n$ for some $h\in\C[X]\setminus X\C[X]$
  if and only if
  there exists $x_0\in\C$ such that $e_f(x_0)=s$ and $n\mid e_f(x_1)$
  for every $x_1\ne x_0$ satisfying $f(x_0)=f(x_1)$.
\end{lemma}

We do not give a ramification characterization of $X^s h(X^n)$.  Instead we
characterize these polynomials in a different way in Lemma~\ref{Halbchar}.

\begin{remark}
  Variants of Lemma~\ref{abhy} are classical, but we know no reference for this
  version; for instance, a special case is proved in \cite{Ritt}.
  We know of three other proofs of Lemma~\ref{ramchar}; here we only discuss
  the most difficult part, where we assume that $f$ has
  precisely two finite branch points and every finite ramification point has
  ramification index $2$.  Ritt \cite[p.~65]{Ritt} argued as follows:
  since $\Mon(f)$ has an $n$-cycle, one can show there is only one possibility
  for the permutation
  representations induced by generators of the inertia groups in the Galois
  closure of $\pi_f$, so by topological considerations there is just one
  equivalence class of such polynomials $f$, whence $f$ is dihedral since $T_n$
  has the prescribed ramification.  Versions of this argument appear in
  \cite[Lemma~9]{Fried-Schur} and \cite[Prop.~4]{Tortrat}.  Levi
  \cite[\S 13]{Levi}
  proves this result by observing that, after composing with linears, we have
  $n^2(f^2-4)=(X^2-4)f'(X)^2$; but $\pm T_n$ solve this differential equation,
  so one can deduce that $f$ is dihedral by showing there are
  at most two solutions in degree-$n$ polynomials.
  Beginning with $f^2-4=(X^2-4)h^2$, Dorey and Whaples \cite[p.~97]{DW} factor
  $f-2$ and $f+2$; upon substituting $X=Y+Y^{-1}$ and subtracting
  the expression for $f-2$ from that for $f+2$, they find that $4Y^n$ is the
  difference between the squares of two degree-$n$ polynomials, which
  determines the polynomials and consequently the form of $f$.
  The advantages of our proof are that it uses similar methods to the rest of
  this paper, and also that
  Lemma~\ref{quadext} provides additional information which does not follow
  from these other proofs.
\end{remark}


\subsection{Monodromy groups}
We now show that $X^n$ and $T_n$ are uniquely determined (up to
equivalence) by their monodromy groups.

\begin{lemma} \label{cycdi}
  Pick $f\in\C[X]$ of degree $n>1$, and put $G:=\Mon(f)$.  Then $G$ is cyclic
  if and only if $f$ is cyclic, in which case $\abs{G}=n$.  Likewise, if $n>2$
  then $G$ is dihedral if and only if $f$ is dihedral, in which case
  $\abs{G}=2n$.
\end{lemma}

\begin{proof}
  Since $G$ contains an $n$-cycle, if $G$ is cyclic then $\abs{G}=n$, and if
  $G$ is dihedral and $n>2$ then $\abs{G}=2n$.  The result now follows from
  Lemma~\ref{quadext}.
\end{proof}

\begin{remark}
  A different proof is given in \cite[Thm.~3.8]{Bilu}, using the fact that
  if the multiplicities of the roots of $f(X)-x_0$ are $e_1,\dots,e_k$ then
  $G$ has an element whose cycle lengths are $e_1,\dots,e_k$.  There are only a
  few possibilities for the cycle structure of an element of a dihedral group,
  and in combination with the Riemann--Hurwitz formula for $\pi_f$ this implies
  that if $G$ is dihedral (and $n\ne 4$) then $f$ has precisely two finite
  branch points and every finite ramification point has ramification index $2$.
\end{remark}

\subsection{Decompositions}

We now determine all decompositions of the special polynomials
$X^n$, $T_n$, $X^s h(X^n)$, and $X^s h(X)^n$.
First, $X^n=X^k\circ X^{n/k}$ and
$T_n=T_k\circ T_{n/k}$ for any divisor $k$ of $n$, 
and (\ref{multipledeg}.3) implies these are the only
decompositions of $X^n$ and $T_n$ up to equivalence:

\begin{lemma} \label{chebcheb}
  If $a,b\in\C[X]\setminus\C$ satisfy $a\circ b=X^n$, then $a=X^k\circ\ell$ and
  $b=\ell\iter{-1}\circ X^{n/k}$ for some linear $\ell\in\C[X]$.
  If $a,b\in\C[X]\setminus\C$ satisfy $a\circ b=T_n$, then $a=T_k\circ\ell$ and
  $b=\ell\iter{-1}\circ T_{n/k}$ for some linear $\ell\in\C[X]$.
\end{lemma}

Conversely, we now describe which compositions of cyclic polynomials
are cyclic, and likewise for dihedral polynomials.

\begin{lemma} \label{chebswap}
  If $a$ and $b$ are cyclic, then $a\circ b$ is cyclic if and only if the
  finite ramification point of $a$ equals the finite branch point of $b$.  If
  $a$ and $b$ are dihedral, then $a\circ b$ is dihedral if and only if the
  special points of $a$ coincide with the finite branch points of $b$.
\end{lemma}

\begin{proof}
  This is an immediate consequence of Lemma~\ref{ramchar} and the
  multiplicativity of ramification indices.
\end{proof}

In practice, this result is often used in the following explicit form.

\begin{cor} \label{chebswapcor}
  Pick integers $m,n$ and linear $\ell,\ell_1,\ell_2\in\C[X]$.
  If $m,n>1$ and $X^m\circ \ell\circ X^n = \ell_1\circ X^{mn}\circ\ell_2$,
  then $\ell=\alpha X$ for some $\alpha\in\C^*$.
  If $m,n>2$ and $T_m\circ \ell\circ T_n=\ell_1\circ T_{mn}\circ\ell_2$,
  then $\ell = \epsilon X$ for some $\epsilon\in\{1,-1\}$.
\end{cor}

Now we address the same question for polynomials of the forms $X^s h(X^n)$
or $X^s h(X)^n$.  We first observe that polynomials of these forms behave well
under composition: $X^s h(X^n) \circ X^{\hat s} {\hat h}(X^n) =
X^{s \hat s} \tilde{h}(X^n)$ where
$\tilde{h}(X)={\hat h}(X)^s h(X^{\hat s} {\hat h}(X)^n)$, and likewise
$X^s h(X)^n\circ X^{\hat s} {\hat h}(X)^n = X^{s\hat s} \tilde{h}(X)^n$.
Conversely, we now show that these are the only ways that polynomials of these
forms can decompose.

\begin{lemma} \label{decompositions}
  Pick $a,b,h\in\C[X]\setminus\C$ and coprime positive integers $s$ and $n$.
  If $a\circ b=X^s h(X)^n$ then $a=X^j \hat{h}(X)^n\circ\ell$ and
  $b=\ell\iter{-1}\circ X^k \tilde{h}(X)^n$ for some $j,k>0$ and some
  $\hat{h},\tilde{h},\ell\in\C[X]$ with $\ell$ linear.  If
  $a\circ b = X^s h(X^n)$ then $a=X^j\hat{h}(X^n)\circ\ell$ and
  $b=\ell\iter{-1}\circ X^k\tilde{h}(X^n)$ for some $j,k>0$ and some
  $\hat{h},\tilde{h},\ell\in\C[X]$ with $\ell$ linear.
\end{lemma}

\begin{proof}
  Suppose $a\circ b = X^s h(X)^n$.  After replacing $a$ and $b$ by
  $a\circ\ell\iter{-1}$ and $\ell\circ b$ for some linear $\ell\in\C[X]$, we
  may assume $a(0)=b(0)=0$ and $a$ is monic.  Write
  $a=X^j\prod_{\beta} (X-\beta)^{n_{\beta}}$, where $\beta$ varies over the
  distinct nonzero roots of $a$.  Since the various polynomials $b-\beta$ are
  coprime to one another and to $b$, it follows that $b^j$ equals $X^s$ times
  an $n\tth$ power.  But $j$ divides $e_{a\circ b}(0)$, which is coprime to
  $n$, so $\gcd(j,n)=1$ and thus
  $b=X^k \tilde{h}(X)^n$ for some $\tilde{h}\in\C[X]$.  Every $b$-preimage
  of $\beta$ has ramification index divisible by $n/\gcd(n,n_{\beta})$;
  if $n\nmid n_{\beta}$ then this yields too large a contribution to the
  Riemann--Hurwitz formula for $b$.  Thus $a=X^j \hat{h}(X)^n$ for some
  $\hat{h}\in\C[X]$.

  Now suppose $a\circ b = X^s h(X^n)$.  For any primitive $n\tth$ root of
  unity $\zeta$, we have $a(b(\zeta X))=\zeta^s a(b(X))$, so 
  (\ref{multipledeg}.3) implies $a = \zeta^s a\circ\ell_3$ and
  $b(\zeta X)=\ell_3\iter{-1}\circ b$ for some linear $\ell_3\in\C[X]$.
  Thus $b=\beta + X^k\tilde{h}(X^n)$ for some $\beta\in\C$, so replacing
  $a$ and $b$ by $a(X+\beta)$ and $b-\beta$ implies $\ell_3=\zeta^k X$.
  Since $\deg(a\circ b)$ is coprime to $n$, we have $\gcd(k,n)=1$, so
  $\hat\zeta:=\zeta^k$ is a primitive $n\tth$ root of unity, and we conclude
  from $a=\zeta^s a\circ \hat\zeta X$ that $a=X^j \hat{h}(X^n)$.
\end{proof}

\begin{remark}
The fact that odd polynomials only decompose into odd polynomials
was proved in \cite[Prop.~1]{HR}; the analogous fact for decompositions
of $X^s h(X^n)$ with $n$ prime is \cite[Thm.~4.3]{H}.
\end{remark}


\subsection{Equivalence} \label{equivalence}

We now determine all equivalences between polynomials of the forms
$X^n$, $T_n$, $X^s h(X^n)$, and $X^s h(X)^n$.  This enables us to describe all
Ritt moves involving any prescribed polynomial.

\begin{lemma} \label{chebeq}
  If $n>1$ and $\ell_1,\ell_2\in\C[X]$ satisfy
  $\ell_1\circ X^n\circ \ell_2 = X^n$, then
  $\ell_2=\alpha X$ and $\ell_1=X/\alpha^n$ for some $\alpha\in\C^*$.
  If $n>2$ and $\ell_1,\ell_2\in\C[X]$ satisfy
  $\ell_1\circ T_n\circ \ell_2 = T_n$, then $\ell_2=\epsilon X$ and
  $\ell_1=\epsilon^n X$ for some $\epsilon\in\{1,-1\}$.
\end{lemma}

\begin{proof}
  In either case, comparing degrees gives $\deg(\ell_1)=\deg(\ell_2)=1$, and
  comparing coefficients of $X^{n-1}$ implies $\ell_2(0)=0$, so
  $\ell_2=\alpha X$ with $\alpha\in\C^*$.
  If $\ell_1\circ X^n\circ \alpha X=X^n$ then $\ell_1=X/\alpha^n$.  Now
  suppose that $n>2$ and $\ell_1\circ T_n\circ \alpha X=T_n$.  Since the ratio
  of the coefficients of $X^n$ and $X^{n-2}$ in $\ell_1\circ T_n\circ \alpha X$
  is $\alpha^2$ times the corresponding ratio in $T_n$, we have
  $\alpha\in\{1,-1\}$.  Since $T_n(-X)=(-1)^n T_n(X)$, this implies
  $\ell_1=\alpha^n X$.
\end{proof}

\begin{lemma} \label{cyceqdi}
  The polynomials $T_n$ and $X^n$ are equivalent if and only if $n\le 2$.
\end{lemma}

\begin{proof}
  This follows from Lemma~\ref{cycdi}, but we give an alternate proof.  For
  $n\le 2$, there is a unique equivalence class of degree-$n$ polynomials.  Now
  suppose $n>2$ and $T_n = \ell_1\circ X^n\circ \ell_2$ with
  $\ell_1,\ell_2\in\C[X]$ linear.
  Since the coefficient of $X^{n-1}$ is zero in both $T_n$ and $X^n$, we must
  have $\ell_2(0)=0$.  But then the coefficient of $X^{n-2}$ in
  $\ell_1\circ X^n\circ \ell_2$
  is zero, yet the coefficient of $X^{n-2}$ in $T_n$ is nonzero, contradiction.
\end{proof}

\begin{lemma} \label{twistcheb}
  Pick $n,s>0$ and $h\in\C[X]$, and let $f$ be either $X^s h(X)^n$ or
  $X^s h(X^n)$.  If $f$ is cyclic and $n>1$ then $h$ is a monomial.
  If $f$ is dihedral then $n\le 2$.
\end{lemma}

\begin{proof}
  If $f=X^s h(X)^n$ is equivalent to $X^k$ with $k>1$, then the unique finite
  branch point of $f$ has just one $f$-preimage.  If also $n>1$ then each
  nonzero root of $h$ is a ramification point of $f$ having the same $h$-image
  as $X=0$, a contradiction; thus $h$ has no nonzero roots, so $h$ is a
  monomial.
  
  If $f=X^s h(X)^n$ is equivalent to $T_k$ with $k>2$, then each ramification
  point of $f$ has ramification index $2$; thus $s\le 2$, so $h$ is
  non-constant, and each root $\alpha$ of $h$ satisfies $e_f(\alpha)\ge n$
  so $n\le 2$.

  Suppose $\ell_1\circ X^k\circ\ell_2 = X^s h(X^n)$ with $k>1$ and
  the $\ell_i$ linear.  If $n>1$ then equating coefficients of $X^{k-1}$
  gives $\ell_2(0)=0$, so evaluating at $X=0$ gives $\ell_1(0)=0$, whence
  $h$ is a monomial.
  
  Suppose $\ell_1\circ T_k\circ\ell_2 = X^s h(X^n)$ with $k>2$ and the
  $\ell_i$ linear.  If $n>2$ then the coefficients of $X^{k-1}$ and $X^{k-2}$
  on the right side are zero, but it is not possible for the corresponding
  coefficients on the left side to both be zero.
\end{proof}

We now describe the Ritt moves involving at least one dihedral polynomial.
Here the crucial point is that if such a move has type (\ref{Ritt2}.2)
then it can be rewritten as a move of type (\ref{Ritt2}.1).

\begin{lemma} \label{chebritt}
  Suppose $n>2$ and $a\circ b = c\circ d$ where $a,b,c,d\in\C[X]\setminus\C$
  satisfy $\gcd(\deg(a),\deg(c))=\gcd(\deg(b),\deg(d))=1$.  If $c=T_n$ then
  $d=\epsilon T_m\circ\ell$ and $a=\epsilon^n T_m\circ\hat\ell$ and
  $b=\hat\ell\iter{-1}\circ T_n\circ\ell$ where $\ell,\hat\ell\in\C[X]$ are
  linear and $\epsilon\in\{1,-1\}$.  If $d=T_n$ then
  $a=\ell\circ T_n\circ\hat\ell$ and
  $b=\hat\ell\iter{-1}\circ \epsilon T_m$ and $c=\ell\circ \epsilon^n T_m$
  where $\ell,\hat\ell\in\C[X]$ are linear and $\epsilon\in\{1,-1\}$.
\end{lemma}

\begin{proof}
  First suppose $c=T_n$.  Since $n>2$, Lemma~\ref{cyceqdi} implies $c$ is not
  cyclic.  By Theorem~\ref{Ritt2}, there are linear $\ell_j\in\C[X]$ for
  which the quadruple $Q:=(\ell_1\circ a\circ\ell_2,\,
  \ell_2\iter{-1}\circ b\circ\ell_3,\, \ell_1\circ c\circ\ell_4,\,
  \ell_4\iter{-1}\circ d\circ\ell_3)$ has one of the forms
  (\ref{Ritt2}.1) or (\ref{Ritt2}.2).  If it is (\ref{Ritt2}.1), then
  Lemma~\ref{chebeq} implies $\ell_4=\epsilon X$ and $\ell_1=\epsilon^n X$
  for some $\epsilon\in\{1,-1\}$, and the result follows.  So assume $Q$ has
  the form (\ref{Ritt2}.2); we will show that, after perhaps changing the
  $\ell_j$'s, we can also write $Q$ in the form (\ref{Ritt2}.1).
  Now $c=\ell_1\iter{-1}\circ X^s h(X)^N\circ\ell_4\iter{-1}$ where
  $h\in\C[X]$,\, $s\ge 0$, and $N:=\deg(a)$.
  Lemma~\ref{twistcheb} implies $N\le 2$.
  If $N=1$ then $a$ and $d$ are linear, so
  $a\circ b=c\circ d$ can be written in
  the form of (\ref{Ritt2}.1) as
  $(T_1\circ a) \circ (a\iter{-1}\circ T_n\circ d) = T_n \circ (T_1 \circ d)$.
  Now assume $N=2$, so $n$ is odd (since $\gcd(\deg(a),\deg(c))=1$).
  Since each ramification point of $T_n$ has ramification index $2$, we must
  have $s=1$; thus $X=0$ is a special point of $s^i h(X)^N$, so $\ell_4(0)$
  is a special point of $T_n$ and hence equals $2\epsilon$ for some
  $\epsilon\in\{1,-1\}$.  Thus $\ell_4=\alpha X+2\epsilon$ where
  $\alpha\in\C^*$.  Now
  $\ell_4\circ X^2=\alpha X^2+2\epsilon=-\epsilon T_2(\gamma X)$ where
  $\gamma^2=-\epsilon\alpha$, so $d=-\epsilon T_2\circ \gamma\ell_3\iter{-1}$.
  Since $T_n(-\epsilon X)=-\epsilon T_n$, it follows that
  $c\circ d = -\epsilon T_{2n}\circ \gamma\ell_3\iter{-1} =
   (-\epsilon T_2)\circ (T_n\circ\gamma\ell_3\iter{-1})$.
  Since also $c\circ d=a\circ b$,
  by (\ref{multipledeg}.3) we have $a=-\epsilon T_2\circ\hat\ell$ and
  $b=\hat\ell\iter{-1}\circ T_n\circ\gamma\ell_3\iter{-1}$ for some linear
  $\hat\ell$.  Combined with the expressions
  $c=-\epsilon T_n\circ (-\epsilon X)$ and
  $d=-\epsilon T_2\circ\gamma\ell_3\iter{-1}$, this shows that (after perhaps
  changing the $\ell_j$'s) the quadruple $Q$ can be written in the form
  (\ref{Ritt2}.1).
  
  One can use a similar (but easier) argument to prove the result when $d=T_n$;
  alternately, Theorem~\ref{monthm} and Lemma~\ref{cycdi} imply $a$ is
  dihedral, so the result follows from what was proved above.
\end{proof}

We now characterize the polynomials $X^s h(X^n)$ in terms of their
self-equivalences.

\begin{notation}
  For $f\in\C[X]$ with $\deg(f)>1$, let $\Gamma(f)$ be the set of linear
  $\ell\in\C[X]$ for which there exists $\hat\ell\in\C[X]$ with
  $\hat\ell\circ f = f\circ\ell$.
\end{notation}

Note that $\Gamma(f)$ is closed under composition and inversion, and hence
is a group under composition.  Further, $\Gamma(f)$ contains the group
$\Gamma_0(f)$ defined in Remark~\ref{BNrem}.

\begin{lemma} \label{Halbchar}
  Pick $f\in\C[X]$ with $\deg(f)=k>1$.  Then $\Gamma(f)$ is infinite if and
  only if there are linear $\ell_1,\ell_2\in\C[X]$ for which
  $\ell_1\circ f\circ\ell_2=X^k$, in
  which case $\Gamma(f)=\{\ell_2\circ \alpha\ell_2\iter{-1}:\alpha\in\C^*\}$.
  Also
  $\abs{\Gamma(f)}=n>1$ if and only if there are linear $\ell_1,\ell_2\in\C[X]$
  for which $\ell_1\circ f\circ \ell_2 = X^s \hat f(X^n)$ where $s\ge 0$ and
  $\hat f\in\C[X]$ is neither a monomial nor a polynomial in $X^j$ for any
  $j>1$.  In this case
  $\Gamma(f)=\{\ell_2\circ \alpha \ell_2\iter{-1}:\alpha^n=1\}$ is cyclic.
\end{lemma}

\begin{proof}
  Pick linear $\ell_1,\ell_2\in\C[X]$ such that
  $g:=\ell_1\circ f\circ\ell_2$ is monic and
  has no terms of degrees $k-1$ or $0$.  If $g\ne X^k$ then
  $\Gamma(g) =  \{\alpha X:\alpha^n=1\}$ where $n$ is the
  greatest common divisor of the differences between degrees of terms of $g$,
  so $\Gamma(g)$ has order $n$ where $g=X^s \hat f(X^n)$ with $s,\hat f$ as
  required.  Since $\Gamma(X^k)=\{\alpha X:\alpha\in\C^*\}$ and
  $\Gamma(f)=\ell_2\circ\Gamma(g)\circ\ell_2\iter{-1}$, the result follows.
\end{proof}

Much of Lemma~\ref{Halbchar} was proved in \cite[\S 3]{Halbpol}.
This result allows us to determine all decompositions of even polynomials:

\begin{cor} \label{even}
  For $f,g\in\C[X]\setminus\C$ and $n>1$, we have $f\circ g\in\C[X^n]$ if and
  only if $f=\hat f(X^{n/\gcd(n,s)})\circ\ell\iter{-1}$ and
  $g=\ell\circ X^s \hat g(X^n)$
  for some $r\ge 0$ and some $\hat f,\hat g,\ell\in\C[X]$ with $\ell$ linear.
\end{cor}


\begin{proof}
  Let $\zeta$ be a primitive $n\tth$ root of unity.
  If $f\circ g\in\C[X^n]$ then $f\circ g = f\circ g(\zeta X)$, so by
  Corollary~\ref{multipledeg} we have $g=\tilde\ell\circ g(\zeta X)$ for some
  linear $\tilde\ell$.  Thus $\zeta X\in\Gamma(g)$, so Lemma~\ref{Halbchar}
  implies $g=\ell\circ X^s \hat g(X^n)\circ \hat\ell$ where $s\ge 0$ and
  $\hat g,\ell,\hat\ell\in\C[X]$ with $\ell,\hat\ell$ linear; moreover,
  $\zeta X = \hat\ell\circ \hat\zeta \hat\ell\iter{-1}$ for some
  $\hat\zeta\in\C^*$, whence $\hat\zeta=\zeta$ and $\hat\ell=\gamma X$.  Now
  $\ell\iter{-1}\circ g(\zeta X)=\zeta^s\ell\iter{-1}\circ g$, so
  $f\circ g = f\circ g(\zeta X) =
   f\circ \ell\circ \zeta^s\ell\iter{-1}\circ g$.
  Thus $f=f\circ\ell\circ\zeta^s\ell\iter{-1}$, so
  $f\circ\ell=(f\circ\ell)\circ\zeta^s X$, whence
  $f\circ\ell\in\C[X^{n/\gcd(n,s)}]$.
\end{proof}

\begin{remark}
  Corollary~\ref{even} was proved for $n=2$ in \cite[Prop.~1 and Thm.~1]{HR},
  and for prime $n$ in \cite[Thm.~4.3]{H}.
\end{remark}

We now determine the equivalences between polynomials of the form $X^s h(X^n)$.
Note that a polynomial can be written in this form with different values of
$s$, $n$, and $h$; we now show that composing with linears does not introduce
any essentially different expressions of this form.

\begin{lemma} \label{gatherez}
  Suppose $f:=X^s h(X^n)$ and $g:=X^r \hat{h}(X^m)$ satisfy
  $f=\ell_1\circ g\circ\ell_2$, where $h,\hat{h}\in\C[X]\setminus X\C[X]$,
  the $\ell_i$ are linear, and $m,n>1$ and $r,s>0$.  Then $r=s$ and
  $h\in\C[X^{m/\gcd(m,n)}]$, and moreover if $f$ is nonlinear then
  $\ell_1=\gamma X$ and $\ell_2=\alpha X$ with $\alpha,\gamma\in\C^*$.
\end{lemma}

\begin{proof}
  We may assume $\deg(f)>1$, since otherwise the conclusion visibly holds.
  Since neither $f$ nor $g$ has a term of degree $(\deg(f)-1)$ or $0$, we must
  have $\ell_2=\alpha X$ and $\ell_1=\gamma X$ with $\alpha,\gamma\in\C^*$.
  Thus $X^s h(X^n)=\gamma\alpha^r X^r \hat{h}(\alpha^m X^m)$, so $r=s$ and
  $h(X^n)\in\C[X^n]\cap\C[X^m]=\C[X^{\lcm(m,n)}]$, which implies the result.
\end{proof}

There can be nontrivial equivalences between polynomials of the form
$X^s h(X)^n$: for instance, $X^2 (X+1)^3 = X^3 (X-1)^2 \circ (X+1)$.  We now
give a presentation of a polynomial which displays all such equivalences.

\begin{lemma} \label{together}
  Pick $h\in\C[X]\setminus X\C[X]$ and coprime $n>1$ and $s>0$,
  and suppose that $f:=X^s h(X)^n$ is neither linear nor cyclic nor dihedral.
  Then $f=\tilde h(X)^{mq}\prod_{i=1}^k (X-\beta_i)^{mr_iq/q_i}$ where
  $\tilde h\in\C[X]$,\, $m,r_i>0$,\, $q_i>1$,\, $q=\prod_{i=1}^k q_i$,\,
  $\gcd(q_i,r_iq/q_i)=1$,\, the $\beta_i\in\C$ are distinct, and
  $\tilde h(\beta_i)\ne 0$.  Moreover, there is an expression of $f$ in this
  form for which the following holds:  for any linear $\ell_1,\ell_2\in\C[X]$
  such that $\ell_1\circ f\circ\ell_2=X^{\hat s} \hat{h}(X)^{\hat n}$ with
  $\hat{h}\in\C[X]\setminus X\C[X]$ and coprime $\hat n>1$ and $\hat s>0$,
  there exists $i$ such that $\ell_1=\gamma X$ and $\ell_2=\alpha X+\beta_i$
  with $\gamma,\alpha\in\C^*$, where $\hat s=mr_i q/q_i$ and $\hat n\mid q_i$.
\end{lemma}

\begin{proof}
  Let $S$ be the set of roots of $f$, so $S$ consists of $0$ and the set of
  roots of $h$.  Since $f$ is neither linear nor cyclic, $h$ is nonconstant, so
  since $h\notin X\C[X]$ it follows that $h$ has nonzero roots.  Each nonzero
  element of $S$ is a ramification point of $f$ with ramification index
  divisible by $n$; also, $e_f(0)=s$.  Put $m:=\gcd(e_f(\beta):\beta\in S)$,
  and let $\beta_1:=0, \beta_2,\dots,\beta_k$ be the elements of $S$ for
  which $q_i:=\gcd(e_f(\beta)/m:\beta\in S\setminus\{\beta_i\})$ satisfies
  $q_i>1$.  Write $R_i:=e_f(\beta_i)/m$, so $\gcd(R_i,q_i)=1$.
  Since $q_i\mid R_j$ for $i\ne j$, we must have $\gcd(q_i,q_j)=1$.  Putting
  $q:=\prod_{i=1}^k q_i$, it follows that
  $f=\tilde h(X)^{mq} \prod_{i=1}^k (X-\beta_i)^{mr_iq/q_i}$ for some
  $\tilde h\in\C[X]$, where $r_i:=R_i/\prod_{j\ne i} q_j$ is a positive integer
  coprime to $q_i$.

  The roots of $f$ contribute at least
  $(\deg(f)-1)/2$ to the Riemann--Hurwitz formula for the cover
  $\pi_f\colon\Line\to\Line$ corresponding to $f$, and if equality holds then
  $\deg(f)$ is odd and every root has multiplicity at most $2$.  Pick linear
  $\ell_1,\ell_2\in\C[X]$ such that
  $\ell_1\circ f\circ\ell_2=X^{\hat s} \hat{h}(X)^{\hat n}$, with
  $\hat{h}\in\C[X]\setminus X\C[X]$ and coprime $\hat n>1$ and $\hat s>0$.
  Then the preimages of $\ell_1\iter{-1}(0)$ under $f$ also contribute at
  least $(\deg(f)-1)/2$ to the Riemann--Hurwitz formula for $\pi_f$.
  But the sum of the Riemann--Hurwitz
  contributions at all finite values is $\deg(f)-1$, so if
  $\ell_1\iter{-1}(0)\ne 0$ then $f$ has precisely two finite branch points,
  and every finite ramification point has ramification index $2$, whence $f$
  is dihedral (by Lemma~\ref{ramchar}), contradiction.  Thus $\ell_1=\gamma X$
  for some $\gamma\in\C^*$.  Next, $\hat n\mid e_f(\beta)$ for every 
  $\beta\in S\setminus\{\ell_2(0)\}$, but $\hat n$ is coprime to
  $\hat s=e_f(\ell_2(0))$,
  so $\ell_2(0)=\beta_i$ and $\hat n\mid q_i$ for some $i$.
  Thus $\ell_2 = \alpha X+\beta_i$ for some $\alpha\in\C^*$, so
  $\gamma f(\alpha X+\beta_i) = X^{\hat s} \hat{h}(X)^{\hat n}$.  Finally,
  equating the ramification indices of $X=0$ on both sides gives
  $\hat s=mr_i q/q_i$.
\end{proof}

Finally, we determine equivalences between $X^s h(X)^n$ and
$X^r \hat{h}(X^m)$.

\begin{lemma}\label{mix}
  Suppose $f:=X^s h(X^n)$ and $g:=X^r \hat{h}(X)^m$ satisfy
  $g=\ell_1\circ f\circ\ell_2$, where $h,\hat{h}\in\C[X]\setminus X\C[X]$,
  the $\ell_i\in\C[X]$ are linear, and $m,n>1$ and $r,s>0$ are such that
  $\gcd(r,m)=\gcd(s,n)=1$.  If $f$ is not linear or dihedral then $r=s$ and
  $h=h_0^m$ for some $h_0\in\C[X]$, and moreover $\ell_1=\gamma X$
  and $\ell_2=\alpha X$ with $\alpha,\gamma\in\C^*$.
\end{lemma}

\begin{proof}
  Since $\Gamma(f)$ contains $\zeta X$ where $\zeta$ is a primitive $n\tth$
  root of unity, $\Gamma(g)$ contains $\ell:=\ell_2\iter{-1}\circ\zeta\ell_2$.
  Here $\ell=\zeta X+\delta$ with $\delta\in\C$, and there is a linear
  $\hat\ell\in\C[X]$ for which
  \[
    X^r \hat{h}(X)^m \circ \ell = \hat\ell\circ X^r \hat{h}(X)^m.
  \]
  If $f$ is not cyclic then Lemma~\ref{together} implies $\hat\ell(0)=0$;
  by equating the roots of multiplicity $s$ on the two sides of the
  above equality, we obtain $\ell(0)=0$.
  If $f$ is cyclic then Lemma~\ref{twistcheb} implies $\hat{h}\in\C^*$,
  so by Lemma~\ref{chebeq} we again have $\ell(0)=\hat\ell(0)=0$.
  Thus, in either case,
  $\hat{h}(\zeta X)$ is a scalar times $\hat{h}(X)$.  Since
  these polynomials have the same nonzero constant term, it follows that
  $\hat{h}=\tilde h(X^n)$ for some $\tilde h\in\C[X]\setminus X\C[X]$.
  Since $\deg(g)=\deg(f)$ is coprime to $n$, and $\deg(g)\equiv r\pmod{n}$,
  we must have $\gcd(r,n)=1$.  Now Lemma~\ref{gatherez} implies that
  $\ell_1=\gamma X$ and $\ell_2=\alpha X$ for some $\alpha,\gamma\in\C^*$, and
  also $r=s$.  Thus $h(X^n)$ is a scalar times $\tilde h(X^n)^m$, so $h$ is an
  $m\tth$ power.
\end{proof}

\begin{remark} \label{justtwo}
  If $f$ and $g$ satisfy the hypotheses of Lemma~\ref{mix}, and $f$ is neither
  linear nor cyclic nor dihedral, then we must have $k=1$ in
  Lemma~\ref{together}.  If $f$ satisfies the hypotheses of
  Lemma~\ref{together}, and $f$ is not a nontrivial power of another
  polynomial, then $m=1$ and $k\le 2$; this applies in particular when $f$ is
  indecomposable.  The non-dihedral hypotheses in the previous two lemmas
  cannot be removed: if $n$ is odd then $T_n=X \tilde h(X^2-2)$ for some
  squarefree $\tilde h\in\C[X]$, and consequently $T_n+2=(X+2)\tilde h(X)^2$
  and $T_n-2=(X-2)\tilde h(-X)^2$.  If $f$ is linear or cyclic then the last
  assertion of Lemma~\ref{together} does not hold.
\end{remark}


\section{Combining multiple Ritt moves}
\label{sec constraints}

Pick $f\in\C[X]$ with $\deg(f)>1$, and let $\cU=(u_1,\dots,u_r)$ and
$\cV=(v_1,\dots,v_s)$ be complete
decompositions of $f$.
By Corollary~\ref{degrees}, $s=r$ and the sequence
$(\deg(u_1),\dots,\deg(u_r))$ is a
permutation of $(\deg(v_1),\dots,\deg(v_r))$.  Thus there is a unique
permutation $\sigma=\sigma_{\cU,\cV}$ of $\{1,2,\dots,r\}$ such that both
\begin{itemize}
  \item $\deg(u_i)=\deg(v_{\sigma(i)})$ for $1\le i\le r$; and
  \item if $1\le i<j\le r$ satisfy $\deg(u_i)=\deg(u_j)$ then
    $\sigma(i)<\sigma(j)$.
\end{itemize}
Here $\sigma$ defines a bijection between $\cU$ and $\cV$,
via $\sigma\colon u_i\mapsto v_{\sigma(i)}$.

In this section we use the permutation $\sigma_{\cU,\cV}$ to obtain
information about the shape of $f$.
We begin with a simple observation:

\begin{lemma} \label{coprime lemma}
  If integers $1\le i<j\le r$ satisfy
  $\sigma_{\cU,\cV}(i)>\sigma_{\cU,\cV}(j)$, then
  $\gcd(\deg(u_i),\deg(u_j))=1$.
\end{lemma}

\begin{proof}
  By Theorem~\ref{Ritt1}, there is a finite sequence of complete decompositions
  $\cU=\cU_0$, $\cU_1$, $\dots$, $\cU_m=\cV$ such that $\cU_{k+1}$ is obtained
  from $\cU_k$ by replacing two consecutive indecomposables $a,b$ by two others
  $c,d$ such that $a\circ b=c\circ d$.  In this situation,
  Corollary~\ref{Ritt move lemma} implies that either $\deg(a)=\deg(c)$
  (and $\deg(b)=\deg(d)$) or $\gcd(\deg(a),\deg(b))=\gcd(\deg(c),\deg(d))=1$.
  Pick the minimal $k$ for which $\sigma_{\cU,\cU_k}(i)>\sigma_{\cU,\cU_k}(j)$,
  so $\cU_k$ is obtained from $\cU_{k-1}$ by a Ritt move involving the
  indecomposables of $\cU_{k-1}$ which correspond to $u_i$ and $u_j$, whence
  these indecomposables have coprime degrees.
\end{proof}

We need more notation to state our results.  For $1\le i,j\le r$, define
\begin{align*}
  \LL(\cU,\cV,i,j)&=\{k: 1\le k<i,\,\sigma(k)<\sigma(j)\}; \\
  \LR(\cU,\cV,i,j)&=\{k: 1\le k<i,\,\sigma(k)>\sigma(j)\}; \\
  \RL(\cU,\cV,i,j)&=\{k: i<k\le r,\,\sigma(k)<\sigma(j)\}; \\
  \RR(\cU,\cV,i,j)&=\{k: i<k\le r,\,\sigma(k)>\sigma(j)\}.
\end{align*}
Thus, for instance, $\LR(\cU,\cV,i,j)$ is the set of positions of
indecomposables in $\cU$ which lie to the left of $u_i$, but which correspond
to indecomposables in $\cV$ lying to the right of the indecomposable
corresponding to $u_j$.  We also write
\[
LL(\cU,\cV,i,j) = \prod_{k\in \LL(\cU,\cV,i,j)} \deg(u_k),
\]
and define $LR(\cU,\cV,i,j)$, $RL(\cU,\cV,i,j)$, and
$RR(\cU,\cV,i,j)$ analogously.

\begin{prop} \label{clumps}
  Pick two complete decompositions $\cU=(u_1,\dots,u_r)$ and $\cV$ of some
  polynomial $f\in\C[X]$, and pick $k$ with $1\le k\le r$.  Write
  $LL=LL(\cU,\cV,k,k)$, and define $LR$, $RL$, and $RR$ analogously.
  Then $LR$, $RL$ and\/ $\deg(u_k)$ are pairwise coprime, and
  there exist polynomials
  \begin{align*}
    a& \quad\text{of degree } LL, \\
    d& \quad\text{of degree } RR, \\
    b,\hat{b},\tilde{b},\dot{b}& \quad\text{of degree } LR, \\
    c,\tilde{c},\bar{c},\dot{c}& \quad\text{of degree } RL, \text{ and} \\
    \hat{u},\tilde{u},\bar{u}& \quad\text{indecomposable of the same degree
      as } u_k
  \end{align*}
  such that
  \begin{enumerate}
    \item[(\thethm.1)] $u_1\circ u_2\circ\dots\circ u_{k-1}=a\circ b$ \,and \,
                         $u_{k+1}\circ\dots\circ u_r = c\circ d$;
    \item[(\thethm.2)] $b\circ u_k = \hat{u}\circ\hat{b}$;
    \item[(\thethm.3)] $\hat{u}\circ\hat{b}\circ c =
                         \tilde{c}\circ\tilde{u}\circ\tilde{b}$;
    \item[(\thethm.4)] $u_k\circ c=\bar{c}\circ\bar{u}$;\, \text{ and}
    \item[(\thethm.5)] $b\circ\bar{c}\circ\bar{u} =
                          \dot{c}\circ\tilde{u}\circ\dot{b}$.
  \end{enumerate}
\end{prop}

\begin{proof}
  The coprimality assertions follow from Lemma~\ref{coprime lemma}.  Write
  $p=\deg(u_k)$.  Put $g=u_1\circ\dots\circ u_{k-1}$ and
  $h=u_{k+1}\circ\dots\circ u_r$.  Then $f=g\circ u_k\circ h$ and
  $\deg(g)=LL\cdot LR$ and $\deg(h)=RL\cdot RR$.  Likewise, letting $\tilde{u}$
  denote the indecomposable in $\cV$ corresponding to $u_k$, from $\cV$ we get
  $f=\tilde g\circ \tilde{u}\circ \tilde h$ where $\tilde g,\tilde h\in\C[X]$
  have degrees $LL\cdot RL$ and $LR\cdot RR$, respectively.  By
  Lemma~\ref{gcdlem}, there are $a,b,\hat g\in\C[X]$
  such that $g=a\circ b$ and $\tilde g=a\circ \hat g$, where
  $\deg(a)=\gcd(\deg(g),\deg(\tilde g))$; since $\gcd(LR,RL)=1$, this means
  $\deg(a)=LL$ (so $\deg(b)=LR$).  Likewise, there are $c,d,\hat h\in\C[X]$
  such that $h=c\circ d$ and $\tilde h=\hat h\circ d$, where $\deg(d)=RR$ and
  $\deg(c)=RL$.  This proves (\ref{clumps}.1).
  
  Applying Lemma~\ref{gcdlem} to
  $(a\circ b\circ u_k)\circ (c\circ d) =
   (a\circ \hat g\circ\tilde{u})\circ (\hat h\circ d)$, we obtain
  $a_0,\hat{b}\in\C[X]$ such that $a\circ b\circ u_k = a_0\circ \hat{b}$
  and $\deg(a_0)=\gcd(LL\cdot LR\cdot p, LL\cdot RL\cdot p)=LL\cdot p$, so
  $\deg(\hat{b})=\deg(b)$.  Applying Corollary~\ref{multipledeg} to
  $a\circ (b\circ u_k)= a_0\circ \hat{b}$ gives
  $b\circ u_k = \hat{u}\circ\hat{b}$ for some $\hat{u}\in\C[X]$.
  A complete decomposition of $\hat{b}$ contains no indecomposable of
  degree $\deg(u_k)$ (since $\deg(\hat{b})=LR$ is coprime to $\deg(u_k)$),
  so Corollary~\ref{degrees} implies $\hat{u}$ is indecomposable, which
  proves (\ref{clumps}.2).
  
  Next recall that $f=a\circ (\hat g\circ\tilde{u}\circ \hat h\circ d)$ where
  $\deg(\hat g)=\deg(c)$ and $\deg(\hat h)=\deg(b)$.  Since also
  $f=a\circ(\hat{u}\circ\hat{b}\circ c\circ d)$, by Corollary~\ref{multipledeg}
  there is a linear $\ell\in\C[X]$ such that
  $\ell\circ \hat g\circ\tilde{u}\circ \hat h\circ d =
   \hat{u}\circ\hat{b}\circ c\circ d$.
  Putting $\tilde{c}=\ell\circ \hat g$ and $\tilde{b}=\hat h$, we obtain
  $\tilde{c}\circ\tilde{u}\circ\tilde{b}\circ d =
   \hat{u}\circ\hat{b}\circ c\circ d$.  As above,
  Corollary~\ref{degrees} implies $\tilde{u}$ is indecomposable, which proves
  (\ref{clumps}.3).
  
  Assertions (\ref{clumps}.4) and (\ref{clumps}.5) follow by symmetry.
\end{proof}

Proposition~\ref{clumps} enables us to control the cumulative
effect of a sequence of Ritt moves.  We do this in three results:
Proposition~\ref{mess} addresses the case that some indecomposable is
neither cyclic nor dihedral; Proposition~\ref{movecheb} the case that
some indecomposable is dihedral; and the easier Lemma~\ref{cyclicmove}
handles the case that every indecomposable is cyclic.

\begin{prop} \label{mess}
  Let $\cU=(u_1,\dots,u_r)$ and $\cV$ be two complete decompositions of
  $f\in\C[X]$.  Pick $k$ with $1\le k\le r$, and put
  $n=LR(\cU,\cV,k,k)$ and $m=RL(\cU,\cV,k,k)$.
  If $u_k$ is neither cyclic nor dihedral, then there exist
  $a,d,\tilde{h}\in\C[X]$ and $\beta,\delta\in\C$ and $s\ge 0$ with
  $\gcd(s,mn)=1$ such that
  \begin{gather*}
    u_1\circ\dots\circ u_{k-1} = a\circ X^n\circ (X+\delta)\\
    u_k = (X-\delta)\circ X^s\tilde{h}(X^n)^m\circ (X+\beta)\\
    u_{k+1}\circ\dots\circ u_r = (X-\beta)\circ X^m\circ d.
  \end{gather*}
  In particular, $mn<\deg(u_k)$.
\end{prop}

\begin{proof}
  Let $a,b,c,d,\bar{c},\bar{u},\dot{c},\tilde{u},\dot{b}$ be as in
  Proposition~\ref{clumps}, so $n$ and $m$ are coprime to each other
  and to $\deg(u_k)$.  Since $u_k\circ c=\bar{c}\circ\bar{u}$ and
  $\deg(\bar{c})=\deg(c)=m$ is coprime to $\deg(u_k)$, Theorem~\ref{Ritt2}
  implies (because $u_k$ is not cyclic or dihedral) that
  $u_k=\ell_1\circ X^s h(X)^m\circ\ell_2$ and
  $c=\ell_2\iter{-1}\circ X^m\circ\ell_3$ for some $h\in\C[X]\setminus X\C[X]$,
  some linear $\ell_j\in\C[X]$, and some $s\ge 0$ which is coprime to $m$.
  By replacing $h$ and $\ell_3$ by scalar multiples of themselves, we may
  assume $\ell_1=X-\delta$ and $\ell_2=X+\beta$ with $\beta,\delta\in\C$.
  If $n=1$ then the result follows upon replacing $a$ by
  $a\circ b\circ (X-\delta)$ and $d$ by $\ell_3\circ d$.  A similar argument
  applies if $m=1$, so assume $m,n>1$.
  Since $u_k$ is neither cyclic nor dihedral, Lemma~\ref{chebcheb} implies
  that $u_k\circ c = \bar{c}\circ\bar{u}$ is neither cyclic nor dihedral.
  Now $b\circ (\bar c \circ \bar u) = (\dot c\circ \tilde u)\circ \dot b$,
  and also $\deg(\bar c\circ\bar u)=\deg(\dot c\circ\tilde u)$ is coprime
  to $\deg(b)=n$, so by Theorem~\ref{Ritt2} we have
  $b=\ell_4\circ X^n\circ\ell_5$ and
  $\bar c\circ\bar u=\ell_5\iter{-1}\circ X^{\hat s} \hat h(X^n)\circ\ell_6$
  for some $\hat h\in\C[X]\setminus X\C[X]$, some linear $\ell_j\in\C[X]$, and
  some $\hat s\ge 0$ which is coprime to $n$.  As above, we may assume
  $\ell_5=X+\gamma$ with $\gamma\in\C$.
  Thus
  \[
    (X-\delta)\circ X^{ms} h(X^m)^m\circ \ell_3 = u_k\circ c =
    \bar{c}\circ\bar{u} = (X-\gamma)\circ X^{\hat s} \hat h(X^n)\circ \ell_6,
  \]
  so by Lemma~\ref{gatherez} we have $\gamma=\delta$ and
  $h(X^m)^m\in\C[X^n]$.  Since $\gcd(m,n)=1$, it follows that
  $h\in\C[X^n]$, which gives the result once we replace
  $a$ by $a\circ\ell_4$ and $d$ by $\ell_3\circ d$.
\end{proof}

\begin{prop} \label{movecheb}
  Let $\cU=(u_1,\dots,u_r)$ and $\cV$ be two complete decompositions of some
  $f\in\C[X]$.  Pick $k,i$ such that $1\le k,i\le r$ and $u_k$ is dihedral.
  \begin{enumerate}
    \item[(\thethm.1)] If $i>k$ and $RL(\cU,\cV,i-1,k) > 2$, then
      $u_k\circ u_{k+1}\circ\dots\circ u_i$ is dihedral.
    \item[(\thethm.2)] If $i<k$ and $LR(\cU,\cV,i+1,k) > 2$, then
      $u_i\circ u_{i+1}\circ\dots\circ u_k$ is dihedral.
  \end{enumerate}
\end{prop}

\begin{proof}
  Since the proofs of the two parts are similar, we just give the details
  for (\ref{movecheb}.1).  So assume that $n:=RL(\cU,\cV,i-1,k)$
  satisfies $n>2$.  It suffices to prove the result in case $i$ is chosen as
  large as possible so that this inequality holds (by Lemma~\ref{chebcheb}).
  Thus we may assume $i\in\RL(\cU,\cV,i-1,k)$.
  Write $\hat n:=RL(\cU,\cV,k,k)$, so $n\mid \hat n$.  With notation as in
  Proposition~\ref{clumps},
  we have $u_{k+1}\circ\dots\circ u_r=c\circ d$ where
  $\deg(c)=\hat n$, and also
  $u_k\circ c=\bar{c}\circ\bar{u}$ where $\deg(\bar{c})=\deg(c)$ and
  $\gcd(\deg(c),\deg(u_k))=1$.  Write $u_k=\ell_1\circ T_m\circ\ell_2$ with
  $m>2$ and $\ell_1,\ell_2$ linear.  Then Lemma~\ref{chebritt} implies
  $c=\ell_2\iter{-1}\circ\epsilon T_{\hat n}\circ\ell_3$ for some linear
  $\ell_3$ and some $\epsilon\in\{1,-1\}$.  Let
  $g=u_{k+1}\circ u_{k+2}\circ\dots\circ u_{i-1}$ (and let $g=X$ if $i=k+1$).
  Then $c\circ d=g\circ h$ where $h=u_i\circ\dots\circ u_r$.  By
  Lemma~\ref{gcdlem}, we have $c=a\circ c_0$ and $g=a\circ g_0$, and also
  $d=d_0\circ b$ and $h=h_0\circ b$ and $c_0\circ d_0=g_0\circ h_0$, where
  $a,b,c_0,d_0,g_0,h_0\in\C[X]$ satisfy $\deg(a)=\gcd(\deg(c),\deg(g))$ and
  $\deg(b)=\gcd(\deg(d),\deg(h))$.  Lemma~\ref{chebcheb} implies that
  $a=\ell_2\iter{-1}\circ\epsilon T_s\circ\ell_4$ and
  $c_0=\ell_4\iter{-1}\circ T_{\hat n/s}\circ\ell_3$ for some linear $\ell_4$;
  by replacing $c_0$, $g_0$ and $a$ by $\ell_4\circ c_0$, $\ell_4\circ g_0$ and
  $a\circ\ell_4\iter{-1}$, we may assume $\ell_4=X$.  

  By Lemma~\ref{coprime lemma}, for
  $k+1\le j\le i-1$, if $\gcd(\deg(u_j),n)>1$ then $\sigma(j)<\sigma(k)$;
  since $\hat n/n$ is the product of $\deg(u_j)$ over all
  $j$ for which $k+1\le j\le i-1$ and $\sigma(j)<\sigma(k)$, it follows that
  $\deg(g)/(\hat n/n)$ is coprime to $n$.  Plainly $\hat n/n$ divides
  $\gcd(\deg(g),\hat n)=s$.  Now $s':=s/(\hat n/n)$ divides
  $\deg(g)/(\hat n/n)$, and so is coprime to $n$, and $s'(\hat n/n)=s$ divides
  $\hat n$ so $s'\mid n$, whence
  $s'=1$ and $\hat n=ns$.
  
  By definition, $\gcd(\deg(c_0),\deg(g_0))=\gcd(\deg(d_0),\deg(h_0))=1$ and
  $c_0\circ d_0 = g_0\circ h_0$.  Since
  $c_0=T_{\hat n/s}\circ\ell_3=T_n\circ\ell_3$, Lemma~\ref{chebritt} implies
  that $g_0\circ h_0 = \hat\epsilon T_{\hat m}\circ\ell_5$ for some linear
  $\ell_5\in\C[X]$ and some $\hat\epsilon\in\{1,-1\}$.  Thus
  $u_k\circ g\circ h_0 = u_k\circ a\circ (g_0\circ h_0) = 
  \ell_1\circ T_n\circ\epsilon T_s\circ\hat\epsilon T_{\hat m}\circ\ell_5$
  is dihedral.  Now $n$ is divisible by $\deg(u_i)$ (since $i\in\RL$),
  and $\deg(h_0)=\deg(c_0)=n$,
  so by applying Corollary~\ref{multipledeg} to the
  decompositions $h_0\circ b = h = u_i\circ (u_{i+1}\circ\dots\circ u_r)$
  we find that $h_0=u_i\circ h_1$ for some $h_1\in\C[X]$.  Finally, since
  $u_k\circ g\circ h_0$ is dihedral, Lemma~\ref{chebcheb}
  implies that $u_k\circ g\circ u_i$ is dihedral, as desired.
\end{proof}

\begin{remark}
Proposition~\ref{movecheb} would not be true if we only required that
$RL(\cU,\cV,i-1,k) \ge 2$, since for instance
$T_3\circ (-2+X(X+1)^2)\circ X^2 = T_2\circ T_3\circ X(X^2+1)$
is not dihedral.
\end{remark}

The proof of Proposition~\ref{movecheb} can be adapted to apply when
$u_k$ is cyclic, although it leads to a result with a rather complicated
formulation.
Instead of doing this, we give a result which applies in the one
situation not covered by the previous two results, namely when every $u_k$ is
cyclic.  By Lemma~\ref{chebswap}, the composition $u\circ v$ of cyclic
polynomials is cyclic if and only if the finite branch point of $v$ equals the
finite ramification point of $u$.  Conversely, we now show that if each $u_k$
in a complete decomposition is cyclic, then we can group together blocks of
consecutive $u_k$'s whose composition is cyclic, and any two $u_k$'s
whose relative positions are interchanged via a sequence of Ritt moves
must lie in the same block.

\begin{lemma} \label{cyclicmove}
  Let $\cU=(u_1,\dots,u_r)$ be a complete decomposition of $f\in\C[X]$ in which
  each $u_i$ is cyclic.  Pick $k$ with $1\le k<r$, and suppose the finite
  ramification point of $u_k$ differs from the finite branch point of
  $u_{k+1}$.  Then for any complete decomposition $\cV$ of $f$, and any $i$
  with $1\le i\le r$, we have $\sigma_{\cU,\cV}(j)\le k$ if and only if
  $j\le k$.
\end{lemma}

\begin{proof}
  We first show that, in any Ritt move $u_i\circ u_{i+1} = c\circ d$, the
  composition $u_i\circ u_{i+1}$ is cyclic.  The Ritt move cannot be of type
  (\ref{Ritt2}.1), since in that case $u_i$ and $u_{i+1}$ would be
  equivalent to Chebychev polynomials, and thus would have degree $2$ (by
  Lemma~\ref{cyceqdi}), contradicting the fact that their degrees are coprime.
  Thus the Ritt move is of type (\ref{Ritt2}.2), so
  Lemma~\ref{twistcheb} implies that $u_i\circ u_{i+1}$ is cyclic.  It follows
  that $c$ and $d$ are cyclic, and moreover (by Lemma~\ref{chebswap}) the
  finite branch point of $u_{i+1}$ equals the finite ramification point of
  $u_i$ (so $i\ne k$).  Furthermore, the finite branch point of
  $u_i\circ u_{i+1}$ equals that of both $u_i$ and $c$, and the finite
  ramification point of $u_i\circ u_{i+1}$ equals that of both $u_{i+1}$ and
  $d$.
  
  Let $\cW=(w_1,\dots,w_r)$ be a Ritt neighbor of $\cU$, so $w_j=u_j$ for
  $j\notin\{i,i+1\}$ and $u_i\circ u_{i+1}=w_i\circ w_{i+1}$.  Suppose first
  that $u_i=w_i\circ\ell$ and $u_{i+1}=\ell\iter{-1}\circ w_{i+1}$ for some
  linear $\ell$.  Then $w_i$ and $w_{i+1}$ are cyclic, $u_i$ and $w_i$ have the
  same finite branch point, $u_{i+1}$ and $w_{i+1}$ have the same finite
  ramification point, and if $i=k$ then the finite branch point of $w_{k+1}$
  differs from the finite ramification point of $w_k$.  In the previous
  paragraph we showed that these properties also hold if
  $u_i\circ u_{i+1}=w_i\circ w_{i+1}$ is a Ritt move, in which case we must
  have $i\ne k$.  By Corollary~\ref{Ritt move lemma}, it follows that these
  properties hold in every case.  Thus, in every case, the finite branch point
  of $w_{k+1}$ differs from the finite ramification point of $w_k$, both $w_k$
  and $w_{k+1}$ are cyclic, and $\sigma_{\cU,\cW}(j)\le k$ if and only if
  $j\le k$.  By induction, the same properties hold if $\cU$ and $\cW$ are
  contained in a finite sequence of complete decompositions of $f$ in which any
  two decompositions are Ritt neighbors.
  Thus, the result follows from Theorem~\ref{Ritt1}.
\end{proof}

We can now give our new description of the collection of all complete
decompositions of a polynomial.  We begin with a decomposition
$\cU=(u_1,\dots,u_r)$ of $f$ in which each $u_i$ is either indecomposable
or cyclic or dihedral.  We then move cyclic factors as far to the right as
possible, by the following procedure.  For each $i=1,2,\dots$, do the
following: if $u_i\circ u_{i+1}$ is cyclic or dihedral then replace $u_i$ and
$u_{i+1}$ by $u_i\circ u_{i+1}$ and repeat step $i$.  Otherwise, if
$u_i=g\circ X^m\circ\ell$ with $g\in\C[X]$ and $\ell$ linear and $m>1$ maximal,
and $u_{i+1}=\ell\iter{-1}\circ X^s h(X^n)\circ\hat\ell$ with $\hat\ell$
linear, $h$ nonconstant, $s\ge 0$, and $n$ is maximal, then put $k:=\gcd(n,m)$,
and assume $k>1$.  If $h$ is a monomial then replace $u_i$ and $u_{i+1}$ by
$g\circ X^{m/k}$ and $h(X^n)^k\circ\hat\ell$; otherwise replace $u_i$ and
$u_{i+1}$ by $g\circ X^{m/k}$ and $X^s h(X^{n/k})^k$ and $X^k\circ\hat\ell$
unless $g\circ X^{m/k}$ is linear, in which case replace the new $u_i$ and
$u_{i+1}$ by their composition and repeat step $i$.  Having moved all cyclic
factors to the right, now move some of them to the left as follows.
For each $i=\abs{\cU},\dots,2$ do the following.  If $u_{i-1}\circ u_i$ is
cyclic then replace $u_{i-1}$ and $u_i$ by $u_{i-1}\circ u_i$.
Otherwise make no change except perhaps in case $u_i=\ell\circ X^m\circ g$ with
$g\in\C[X]$ and $\ell$ linear and $m>1$ maximal, and
$u_{i-1}=\hat\ell\circ X^s h(X)^n\circ\ell\iter{-1}$ with $s>0$,
$h\in\C[X]\setminus\C$, and $n$ maximal.  In this case, either make no change
or choose a divisor $k>1$ of $\gcd(m,n)$.  If $u_{i-1}\circ X^k$ is dihedral,
then replace $u_{i-1}$ and $u_i$ by $u_{i-1}\circ X^k$ and $X^{m/k}\circ g$,
unless the latter polynomial is linear in which case replace 
$u_{i-1}$ and $u_i$ by their composition.
Otherwise replace $u_{i-1}$ and $u_i$ by $\hat\ell\circ X^k$,
$X^s h(X^k)^{n/k}$, and $X^{m/k}\circ g$, and repeat step $i$ unless
$X^{m/k}\circ g$ is linear, in which case replace the new $u_i$ and $u_{i+1}$
by their composition.  Finally, expand $\cU$ into a complete decomposition by
replacing each cyclic or dihedral $u_i$ by one of the following types of
complete decompositions: if $u_i=\ell_1\circ X^n\circ\ell_2$ then choose any
permutation $(p_1,\dots,p_s)$ of the prime factors (counted with
multiplicities) of $n$, and replace $u_i$ by $\ell_1\circ X^{p_1}$, $X^{p_2}$,
\dots, $X^{p_s}\circ\ell_2$; and similarly if $u_i$ is dihedral.

The results of this section and the previous section show that this procedure
yields a representative of every equivalence class of complete decompositions
of $f$.  The results of Section~\ref{equivalence} control the different ways
of writing the various polynomials $u_i$ in the forms required in the
procedure.

We now prove a refinement of Theorem~\ref{iterates}.  Here we write $\cZ$ for
the set of polynomials of degree at least $2$ which are equivalent
to either $X^s h(X^n)$ or $X^s h(X)^n$ for some $h\in\C[X]$ and some coprime
positive integers $s,n$ with $n>1$.  Note that $\cZ$ contains $X^m$ for every
$m>1$, and $\cZ$ contains $T_m$ for every odd $m>1$ (and also for $m=2$).
Thus, an indecomposable is in $\cZ$ if and only if it occurs
in a Ritt move.  Recall that $f\iter{k}$ denotes the $k\tth$ iterate of $f$.

\begin{thm} \label{preciseiterates}
  Pick $f\in\C[X]$ with $n=\deg(f)>1$.  Let $a,b\in\C[X]\setminus\C$ and $k>1$
  satisfy $r\circ s = f\iter{k}$, and assume there is no $g\in\C[X]$
  for which either $a=f\circ g$ or $b=g\circ f$.  Let 
  $\cU=(u_1,\dots,u_r)$ be a complete decomposition of $f$.  Then, for each $i$
  with $1\le i\le r$, we have:
  \begin{enumerate}
    \item[(\thethm.1)] If $u_i\notin\cZ$ then $m\le 2$.
    \item[(\thethm.2)] If $k>1$ and $u_i$ is neither cyclic nor dihedral,
      then either
      $n\ge 6\deg(u_i)\ge 6(2^{k-2}+1)$ or
      $n\ge 2\deg(u_i)\ge 2^k+2$.
    \item[(\thethm.3)] If $k>3$ and $u_i$ is dihedral, then
      $f=\ell\circ\epsilon T_n\circ\ell\iter{-1}$ for some linear
      $\ell\in\C[X]$ and some $\epsilon\in\{1,-1\}$.
    \item[(\thethm.4)] If $k>2$ and $u_j$ is cyclic for every $1\le j\le r$,
      then $f=\ell\circ X^n\circ\ell\iter{-1}$ for some linear $\ell\in\C[X]$.
  \end{enumerate}
\end{thm}

\begin{proof}
  Since $f\iter{k}=a\circ b$, by Corollary~\ref{multipledeg} the nonexistence
  of $g$ implies $\deg(f)\nmid\deg(a)$ and $\deg(f)\nmid\deg(b)$.
  Extend $\cU$ to a complete decomposition
  $\cUk=(u_1,\dots,u_{kr})$ of $f\iter{k}$, by putting $u_i=u_{i-r}$ for
  $r+1\le i\le kr$.  Let $\cV=(v_1,\dots,v_{kr})$ be a complete decomposition
  of $f\iter{k}$ such that $a=v_1\circ\dots\circ v_e$ for some $e$.
  The decompositions $\cUk$ and $\cV$ will be implicit in what follows:
  for instance, we will write $\sigma(i)$, $\LR(i,j)$, and $RL(i,j)$ in place
  of $\sigma_{\cUk,\cV}(i)$, $\LR(\cUk,\cV,i,j)$, and $RL(\cUk,\cV,i,j)$.
  
  Since $\deg(f)\nmid\deg(a)$, there is an $I$
  with $1\le I\le t$ such that $\sigma(I)>e$.
  It follows from Lemma~\ref{coprime lemma} that
  $\sigma(I+jr)>e$ for $0\le j<k$.
  Since $\deg(f)\nmid\deg(b)$, there exists $J$ with $1\le J\le r$
  such that $\sigma(J+(k-1)r)\le e$, so
  $\sigma(J+jr)< e$ for $0<j<k-1$.  In particular, $\sigma(J+(k-1)r)<\sigma(I)$,
  so Lemma~\ref{coprime lemma} implies $\deg(u_I)$ and $\deg(u_J)$ are coprime.
  Thus, for $1\le i\le r$, we have $\deg(f)\ge 2\deg(u_i)$; if $u_i$ is neither
  cyclic nor dihedral then $\deg(u_i)\ge 4$, so (\ref{preciseiterates}.2)
  holds for $k=2$.

  Suppose henceforth that $k>2$.  For $1\le i\le kr$, we write
  $\LR(i)$ and $LR(i)$ in place of $\LR(i,i)$ and $LR(i,i)$, and
  we define $\RL(i)$ and $RL(i)$ similarly.  Pick $1\le i\le r$.
  If $\sigma(i)>e$ then
  $J+jr\in\RL(i)$ for $0<j<k$, so $\deg(u_J)^{k-1}\mid RL(i)$;
  in particular, $\deg(u_J)^{k-1}\mid RL(I)$.
  If $\sigma(i+(k-1)r)\le e$ then
  $\deg(u_I)^{k-1}\mid LR(i+(k-1)r)$; thus, $\deg(u_I)^{k-1}\mid LR(J+(k-1)r)$.
  If $\sigma(i)\le e < \sigma(i+(k-1)r)$ then
  there is a unique $m$ with $0\le m<k-1$
  such that $\sigma(i+mr)\le e<\sigma(i+(m+1)r)$.
  Thus $I+jr\in\LR(i+mr)$ for $0\le j<m$, so
  $\deg(u_I)^m \mid LR(i+mr)$; similarly $\deg(u_J)^{k-m-2}\mid RL(i+(m+1)r)$.

  Pick $i$ with $1\le i\le r$.  Proposition~\ref{clumps} implies that
  $u_i\circ c = \bar{c}\circ\bar{u}$ for some $c,\bar{c},\bar{u}\in\C[X]$
  such that $\deg(c)=\deg(\bar{c})$ is coprime to $\deg(u_i)$, where $\deg(c)$
  is the largest element of $\{RL(i+mr): 0\le m<k\}$.
  Similarly, $b\circ u_i = \hat{u}\circ\hat{b}$ for some
  $\bar b,\hat{b},\hat{u}\in\C[X]$ such that $\deg(\bar b)=\deg(\hat{b})$ is
  coprime to $\deg(u_i)$, where $\deg(\bar b)$ is the largest element of
  $\{LR(i+mr): 0\le m<k\}$.
  We showed above that $\bar b$ and $c$ are not both linear; thus
  Theorem~\ref{Ritt2} implies $u_i\in\cZ$, which proves
  (\ref{preciseiterates}.1).  In fact, either
  $\deg(\bar b)\deg(c)\ge \min(\deg(u_I),\deg(u_J))^{k-1} \ge 2^{k-1}$ or
  there is some $m$ with $0\le m<k-1$ such that
  $\deg(\bar b)\deg(c)\ge \deg(u_I)^m \deg(u_J)^{k-m-2}\ge 2^{k-2}$.
  If $\deg(\bar b)\deg(c)<2^{k-1}$ then $i\notin\{I,J\}$, so since
  $\deg(u_I)$ and $\deg(u_J)$ are coprime
  we obtain $\deg(f)\ge \deg(u_i)\deg(u_I)\deg(u_J)\ge 6\deg(u_i)$.

  If $u_i$ is neither cyclic nor dihedral then Theorem~\ref{Ritt2} implies
  $u_i$ is equivalent to both $X^s h(X)^{\deg(c)}$ and
  $X^{\hat s} \hat h(X^{\deg(\overline b)})$,
  where $s,\hat s\ge 0$ and $h,\hat h\in\C[X]\setminus X\C[X]$ satisfy
  $\gcd(s,\deg(c))=1=\gcd(\hat s,\deg(\bar b))$.
  By Lemma~\ref{mix} it follows that $u_i$ is equivalent to
  $X^{\tilde s} \tilde{h}(X^{\deg(\overline b)})^{\deg(c)}$ for some
  $\tilde h\in\C[X]\setminus X\C[X]$
  and some $\tilde s>0$ which is coprime to $\deg(\bar b)\deg(c)$.  Thus
  $\deg(u_i)\ge 1+\deg(\bar b)\deg(c)$, which implies (\ref{preciseiterates}.2).

  Now suppose $u_i$ is dihedral and $k\ge 4$.
  If there is some $j$ with $1\le j\le r$ for which
  $u_j$ is dihedral and $RL(2r,j) > 2$, then
  Proposition~\ref{movecheb} implies $u_j\circ\dots\circ u_{2r+1}$ is dihedral.
  Putting $h:=u_j\circ\dots\circ u_r$, we have
  $h\circ f\circ u_1 = \ell_1\circ T_{mns}\circ\ell_2$ for some linear
  $\ell_1,\ell_2\in\C[X]$, where $m=\deg(h)$ and $s=\deg(u_1)$.  Note that
  $s>1$ and $m>2$ (since $u_j$ dihedral).  By Lemma~\ref{chebcheb}, we have
  \begin{align*}
    h &= \ell_1\circ T_m\circ\ell_3, \\
    f &= \ell_3\iter{-1}\circ T_n\circ\ell_4, \quad\text{ and} \\
    u_1 &= \ell_4\iter{-1}\circ T_s\circ\ell_2\quad \text{ for some linear }
      \ell_3,\ell_4\in\C[X].
  \end{align*}
  Putting $g:=u_1\circ\dots\circ u_{j-1}$, we have
  $g\circ h = f$, so $T_m\circ\ell_3=\ell_5\circ T_m\circ\ell_4$ for
  some linear $\ell_5$.  Now Lemma~\ref{chebeq} implies
  $\ell_4\circ\ell_3\iter{-1}=\epsilon X$ with $\epsilon\in\{1,-1\}$, so
  $f=\ell_4\iter{-1}\circ\epsilon T_n\circ\ell_4$, as desired.
  
  So assume there is no $j$ as above.  Since $\deg(u_I)$ and $\deg(u_J)$ are
  coprime, they cannot both be even; by symmetry, we may assume
  $\deg(u_I)$ is odd.  Since $J+2r,J+3r\in\RL(2r,I)$, our asumption on
  nonexistence of $j$ implies $u_I$ is not dihedral.  This assumption also
  implies $\sigma(i)<e$, since otherwise $J+2r,J+3r\in\RL(2r,i)$.
  If $I<i$ then, since $\sigma(i)<e<\sigma(I)$, (\ref{movecheb}.2) would imply
  $u_I\circ\dots\circ u_i$ is dihedral, which by Lemma~\ref{chebcheb} would
  imply $u_I$ dihedral, contradiction.  Thus $I>i$, and similarly
  $\sigma(I)<\sigma(i+r)$.
  Since $\sigma(J+2r)<\sigma(J+3r)\le e<\sigma(I)<\sigma(i+r)$,
  we have $J+2r,J+3r\in\RL(2r,i+r)$, so Proposition~\ref{movecheb} implies
  $u_{i+r}\circ\dots\circ u_{2r+1}$ is dihedral; since $i+r<I+r<2r+1$, it
  follows that $u_{I+r}$ is dihedral, contradiction.

  Finally, suppose every $u_j$ is cyclic.  Since $\sigma(I)>\sigma(J+2r)$,
  by Lemma~\ref{cyclicmove} the finite ramification point of $u_j$ equals
  the finite branch point of $u_{j+1}$ for $I\le j< J+2r$, and hence also for
  $1\le j<kr$.  Thus $f\iter{k}$ is cyclic (by Lemma~\ref{chebswap}), so
  $f$ is cyclic, whence $f=\ell_1\circ X^n\circ \ell_2$ with $\ell_1,\ell_2$
  linear.  Then the finite ramification point of $u_r$ equals that of $f$,
  namely $\ell_2\iter{-1}(0)$; likewise the finite branch point of
  $u_1=u_{r+1}$ equals that of $f$, namely $\ell_1(0)$.  Since these points
  coincide, $\ell_2\circ\ell_1$ fixes $0$, and so has the form $\alpha X$ with
  $\alpha\in\C^*$, whence $f=\ell_2\iter{-1}\circ \alpha X^n\circ\ell_2$.
  By replacing $\ell_2$ by $\alpha^{1/(1-n)}\ell_2$, we may assume $\alpha=1$,
  proving (\ref{preciseiterates}.4).
\end{proof}

We now give examples showing that the conclusion of
Theorem~\ref{preciseiterates} cannot be improved.

\begin{example}
  The exceptions in (\ref{preciseiterates}.3) and (\ref{preciseiterates}.4)
  cannot be avoided.  
  First, by Corollary~\ref{multipledeg}, if $n$ is a prime power then the
  hypotheses of Theorem~\ref{preciseiterates} cannot hold.
  Now assume $n>1$ is not a prime power.
  For any linear $\ell\in\C[X]$, if we put
  $f:=\ell\circ X^n\circ\ell\iter{-1}$ then
  $f\iter{k} = \ell\circ X^{n^k}\circ\ell\iter{-1}$,
  so if $e>1$ is a prime power dividing $n$ such that $\gcd(e,n/e)=1$, then
  $a:=\ell\circ X^{(n/e)^k}$ and
  $b:=X^{e^k}\circ\ell\iter{-1}$ satisfy $f\iter{k}=a\circ b$ and
  $\deg(f)\nmid\deg(a),\deg(b)$.  Note that $\deg(a),\deg(b)\to\infty$ as
  $k\to\infty$.
  
  Similar remarks apply to $f:=\ell\circ\epsilon T_n\circ\ell\iter{-1}$
  for any linear $\ell\in\C[X]$ and any $\epsilon\in\{1,-1\}$.
\end{example}

\begin{example} \label{ex2}
  The bounds in (\ref{preciseiterates}.2) are best possible.
  For instance, pick an integer $m>1$ and let $f_i:=X(1+X^{2^i})^{2^{m-i}}$ for
  $0\le i\le m$.  Then $X^2\circ f_i = f_{i-1}\circ X^2$, so $f:=f_m\circ X^2$
  and $k:=m+1$ satisfy
  $f\iter{k} = a\circ b$
  for $a:=f_m\circ f_{m-1}\circ\dots\circ f_0$ and $b:=X^{2^{m+1}}$.
  Here $\deg(f)=2^k+2$ does not divide $\deg(a)$ or $\deg(b)$.
  By Lemma~\ref{ramchar}, $f_m$ is neither cyclic nor dihedral,
  since it has more than two finite branch points (because the $2^m$ roots of
  the derivative $f_m'(X)$ have distinct images under $f_m$).
  It follows from (\ref{preciseiterates}.2) that $f_m$ is indecomposable;
  alternately, indecomposability of $f_m$ is equivalent to primitivity of
  $\Mon(f_m)$, which holds because $\Mon(f_m)=S_{1+2^m}$ (as follows from
  Hilbert's theorem on monodromy groups of Morse polynomials, cf.\ 
  \cite[\S III]{Hilbert} or \cite[\S 4.4]{Serre}).
  Likewise, $\hat f:=X^3\circ f_m\circ X^2$ and $\hat k:=m+2$ satisfy
  $\hat f\iter{\hat k} = \hat a\circ \hat b$ where
  $\hat a:=X^3\circ f_m\circ X^3\circ f_{m-1}\circ X^3\circ\dots\circ f_0\circ
   X^3$ and
  $\hat b:=X^2\circ f_0\circ X^{2^{m+1}}$; here
  $\deg(\hat f)=6(2^{\hat k-2}+1)$ does not divide $\deg(\hat a)$ or
  $\deg(\hat b)$.
\end{example}

\begin{example}
  The bound on $k$ in (\ref{preciseiterates}.3) cannot be improved in general.
  Pick coprime odd $e,s>1$, and put
  $f:=X^2\circ (X+2)\circ T_e\circ (X-2)\circ X^s$.
  Then $f\iter{3} = a\circ b$ where
  $a:=X^2\circ (X+2)\circ T_e\circ (X-2)\circ X^4$ and
  $b:=X^s\circ T_e\circ X^s\circ (X+2)\circ T_e\circ (X-2)\circ X^s$.
  Note that $\deg(f)\nmid\deg(a),\deg(b)$, and also $f$ is neither cyclic
  nor dihedral (by Lemmas~\ref{chebcheb} and \ref{cyceqdi}).
\end{example}

\begin{example}
  The bounds on $k$ in (\ref{preciseiterates}.1) and (\ref{preciseiterates}.4)
  cannot be improved.  For instance, pick
  any $g\in\C[X]\setminus\C$ and put $f:=X^2\circ g\circ X^3$; then
  $f\iter{2}=a\circ b$ where $a:=X^2\circ g\circ X^2$ and
  $b:=X^3\circ g\circ X^3$, and plainly $\deg(f)\nmid\deg(a),\deg(b)$.
  If $g=X+1$ then $f$ is the composition of cyclic indecomposables, but $f$ is
  not cyclic (by Lemma~\ref{chebswap}).  The hypotheses of
  (\ref{preciseiterates}.1) are satisfied whenever $f\notin\cZ$, which holds
  for a Zariski-dense sublocus of the locus of polynomials $g$ of any
  prescribed degree greater than $3$; explicitly, $g:=X^4+X^2+X$ is not in
  $\cZ$.
\end{example}

We now deduce Theorem~\ref{iterates} from Theorem~\ref{preciseiterates}.
We need the following simple result.

\begin{lemma} \label{blockiterate}
  Pick $f\in\C[X]$ of degree $n>1$.  Pick $a,b\in\C[X]$ and $e>0$ such that
  $a\circ b=f\iter{e}$.  Then there exist $\hat a,\hat b\in\C[X]$ and
  $i,j,k\ge 0$ such that
  \[
    a=f\iter{i}\circ \hat a \quad\text{ and }\quad
    b=\hat b\circ f\iter{j} \quad\text{ and }\quad
    \hat a\circ \hat b = f\iter{k},
  \]
  and also $\hat a \ne f\circ h$ and $\hat b\ne h\circ b$ for every
  $h\in\C[X]$.
\end{lemma}

\begin{proof}
Let $i,j\ge 0$ be maximal such that $\deg(f)^i\mid\deg(a)$ and
$\deg(g)^j\mid\deg(b)$.  Then Corollary~\ref{multipledeg} implies
$a=g\iter{i}\circ \hat a$ and $b=\hat b\circ g\iter{j}$ for some
$\hat a,\hat b\in\C[X]$.  Thus
$f\iter{e}=a\circ b = f\iter{i}\circ \hat a\circ \hat b\circ f\iter{j}$, so
$f\iter{i}\circ (\hat a\circ \hat b) = f\iter{e-j} = f\iter{i}\circ
 f\iter{e-j-i}$.
If $i=0$ then $a\circ b=f\iter{e-j}$, so the required properties hold since
$\deg(f)$ does not divide $\deg(a)$ or $\deg(b)$.
Henceforth assume $i>0$.  By (\ref{multipledeg}.3), there is a linear
$\ell\in\C[X]$ for which
$\ell\circ (a\circ b) = f\iter{e-j-i}$ and $f\iter{i}\circ \ell = f\iter{i}$.
Upon replacing $a$ by $\ell\circ a$, we obtain the desired conclusion.
\end{proof}

\begin{proof}[Proof of Theorem~\ref{iterates}]
Lemma~\ref{blockiterate} gives everything but the bound on $k$.
So suppose $a\circ b = f\iter{k}$ with $k\ge 0$, where there is no $g\in\C[X]$
for which either $a=f\circ g$ or $b=g\circ f$.  By Corollary~\ref{multipledeg},
neither $\deg(a)$ nor $\deg(b)$ is divisible by $\deg(f)$.  If $f$ is
indecomposable, Corollary~\ref{degrees} implies that each of
$a$ and $b$ is either linear or the composition of indecomposables having
the same degree as $f$; thus $a$ and $b$ must be linear, so $k=0$, whence
$k<\log_2(n)$.
Now let $\cU=(u_1,\dots,u_r)$ be a complete decomposition of $f$, and
assume $r>1$.  Then $n=\deg(f)$ satisfies $n\ge 2\deg(u_i)\ge 4$ for every $i$.
If some $u_i$ is neither cyclic nor
dihedral, then (\ref{preciseiterates}.2) implies $k<\log_2(n)$.
If every $u_i$ is cyclic then (\ref{preciseiterates}.4) implies
$k\le 2\le\log_2(n)$.
Finally, if some $u_i$ is dihedral then (\ref{preciseiterates}.3) implies
$k\le 3$, so $k\le\log_2(n)$ whenever $n\ge 8$.  Since $n$ is composite, the
only possible exceptions are $n=6$ (for which $k=\log_2(n+2)$) and $n=4$.
But if $n=4$ then the degrees of $a$ and $b$
are powers of $2$ which are not divisible by $4$, so $\deg(a),\deg(b)\le 2$
and thus $k\le 1<\log_2(n)$.
\end{proof}

\begin{remark}
The above proof shows that if $n\ne 6$ then the bound on $k$ can be improved to
$k\le\log_2(n)$.  This improvement is not possible for $n=6$, since
$f=T_3\circ 2T_2$ satisfies $f\iter{3}=(T_3\circ 2T_3\circ (4T_3+6)) \circ
(T_4\circ 2T_2)$ (and $f$ is neither cyclic nor dihedral).
\end{remark}


\section{Related topics}
We now briefly discuss some related topics.  First, any polynomial (or rational
function) over any field has only finitely many equivalence classes
of decompositions.  However, in most situations we know much less
about these decompositions than we do in the case of polynomials over $\C$.


\subsection{Decomposition of rational functions}

Ritt \cite{Rittrat,Ritt2} studied decompositions of rational functions over
$\C$.  He recalled \cite[p.~222]{Ritt2} that the groups $A_4$, $S_4$, and
$A_5$ act as groups of automorphisms of $\C(x)$, with fixed field $\C(f)$ where
the equivalence classes of decompositions of $f$ are in bijection with the
(increasing) chains of subgroups of the relevant group.  Since these groups
contain distinct-length maximal chains of subgroups, the
rational function analogue of Theorem~\ref{Ritt1} is not true.
Further examples of distinct-length complete decompositions can be
produced from group actions on the $j=0$ and $j=1728$ elliptic curves.
There are only a few known theorems limiting the possibilities,
the best being Ritt's classification of pairs of commuting rational
functions \cite{Rittrat}.  For the current state of knowledge,
see \cite{LyonsZ}.


\subsection{Decomposition of polynomials over other fields}

All results and proofs in this paper work over arbitrary algebraically closed
fields of characteristic zero.  All but two of our results remain valid over
an arbitrary algebraically closed field whose characteristic does not
divide the degree of the relevant polynomials; the exceptions are
Lemmas~\ref{ramchar} and \ref{chebswap}.
This generalization only presents
difficulties for Theorem~\ref{Ritt2}, where it was done by Zannier
(cf.\ \cite{Zannier} or \cite[\S 1.4]{Schinzel}).
There are versions of all results in this paper (with the above two exceptions)
over any field $K$ whose
characteristic does not divide the degree of the polynomial $f\in K[X]$ under
consideration; this is because in this situation every decomposition of $f$
over the algebraic closure $\bar{K}$ is equivalent to a decomposition over $K$
\cite[\S 2]{Levi}.
(For instance, see \cite[p.~25]{Schinzel} for a version of Theorem~\ref{Ritt2}
in this situation.)
However, new phenomena occur when the characteristic divides $\deg(f)$:
\begin{itemize}
\item An indecomposable polynomial over $K$ can decompose over $\bar{K}$;
  however, Guralnick and Saxl \cite{GS2} proved this can only happen
  for polynomials
  of degree either a power of the characteristic, or $21$ or $55$.  All
  examples of degree $21$ or $55$ were determined in \cite{GZ}.
  Several families of examples of degree a power of the characteristic were
  given in \cite{BealsZ}, in addition to some partial classification results.
\item Two complete decompositions of $f$ can have distinct lengths
  \cite[p.~98]{DW}; see \cite{BealsZ} for further examples, and \cite{BWZ} for
  classes of indecomposables which cannot occur in any such examples.
\item There are decomposable odd polynomials which are not the composition of
  two nonlinear odd polynomials~\cite{BealsZ}.
\end{itemize}

Several of the results from Section~\ref{subsec:Ritt1} remain valid for
decompositions into monic polynomials over any ring in which the degrees
of the polynomials are units.  We will expand on this point elsewhere.


\subsection{Monodromy groups of indecomposable polynomials}

In light of Theorem~\ref{monRitt1}, it is of interest to determine the
possible monodromy groups of indecomposable polynomials.  This was done
in \cite{Feit,Mueller}, according to which the possible groups are cyclic,
dihedral, alternating, symmetric, and finitely many other groups of
small degree.  The analogous problem in positive characteristic is
much more difficult: a reduced list of group-theoretic possibilities
is given in \cite{GS}, and there are families of indecomposable polynomials
whose monodromy groups are quite different from the groups occurring
in characteristic zero (see \cite{Abhyankar,GZ} and the references therein).
The latter families have remarkable properties: for instance, they include
infinite families of pairs $(f,g)$ of non-equivalent indecomposables such
that $f(X)-g(Y)$ is reducible; and also they include several families
of polynomials $f\in\F_q[X]$ for which the map
$\alpha\mapsto f(\alpha)$ induces a bijection on $\F_{q^k}$ for
infinitely many $k$.


\subsection{Algorithms}

Zippel \cite{Zippel} discovered a deterministic polynomial-time algorithm
for finding a complete decomposition of a rational function $f$ over an
arbitrary field $K$.  In case $f$ is a polynomial of degree not divisible by
the characteristic of $K$, the algorithm in \cite{vzG} (following \cite{KL} and
\cite{Levi}) obtains such a decomposition in essentially linear time.
%
By combining this algorithm with Ritt's results, one can compute
representatives of all equivalence classes of complete decompositions of $f$ by
means of $\mathcal{O}(\deg(f)^3)$ arithmetic operations.
Our results yield a faster algorithm, with optimal complexity.
We will present the details elsewhere.


\vskip.2in
\appendix
\section*{Appendix: Ritt's second theorem}

We now prove Theorem~\ref{Ritt2}.

In Section~\ref{sec: Ritt1} we showed that many problems about polynomial
decomposition reduce to questions about subgroups of the inertia group
at infinity.  However, there is no such reduction for the present question:
besides the ramification at infinity, we need to keep track of the ramification
at finite points as well.  The problem amounts to the determination of all
genus-zero curves of the form $a(X)=c(Y)$ with $a,c$ polynomials of coprime
degrees.  We solve it by comparing contributions to the Riemann--Hurwitz
formula for the covers $\Line\to\Line$ corresponding to each of
$a,b,c,d$, where $a\circ b = c\circ d$.

\begin{proof}[Proof of Theorem~\ref{Ritt2}]
  Pick $a,b,c,d\in\C[X]\setminus\C$ such that $a\circ b = c\circ d$ and
  $\gcd(\deg(a),\deg(c))=\gcd(\deg(b),\deg(d))=1$.  Write $m:=\deg(c)$
  and $n:=\deg(a)$, so $\gcd(m,n)=1$ and also $m=\deg(b)$ and $n=\deg(d)$.
  The result is clear if $\min(m,n)=1$, so assume $m,n>1$.
  Let $x$ be transcendental over $\C$, and put $t=a(b(x))$.

  Let $P_1,\dots,P_k$ be the finite
  branch points of $\Line_x\to\Line_t$.  For any $i$ with $1\le i\le k$,
  let $Q_1^{i},\dots,Q_{q(i)}^{i}$ be the points of $\Line_{b(x)}$ lying
  over $P_i$, and let $\alpha_j^{i}$ be the ramification index of
  $Q_j^{i}/P_i$.  Likewise, let $R_1^{i},\dots,R_{r(i)}^{i}$ be the points of
  $\Line_{d(x)}$ lying over $P_i$, and let $\beta_J^{i}$ be the ramification
  index of $R_J^{i}/P_i$.  Then $n=\sum_{j=1}^{q(i)} \alpha_j^{i}$ and
  $m=\sum_{J=1}^{r(i)} \beta_J^{i}$.
  By Lemma~\ref{abhy}, each point $S$ of $\Line_x$
  lying over both $Q_j^{i}$ and $R_J^{i}$ has ramification index
  $\lcm(\alpha_j^{i},\beta_J^{i})$ in $\Line_x\to\Line_t$, and hence has
  ramification index $\lcm(\alpha_j^{i},\beta_J^{i})/\beta_J^{i}$ in
  $\Line_x\to\Line_{d(x)}$.  Moreover, the number of such points $S$ is
  $\gcd(\alpha_j^{i},\beta_J^{i})$.
  Thus, for each $i$,
  some $\alpha_j^{i}$ or $\beta_J^{i}$ is greater than $1$.
  Let $A_i$ and $B_i$ be
  the multisets $\{\alpha_1^{i},\dots,\alpha_{q(i)}^i\}$ and
  $\{\beta_1^{i},\dots,\beta_{r(i)}^i\}$, respectively.  By applying the
  Riemann--Hurwitz formula to the covers $\Line_{b(x)}\to \Line_t$ and
  $\Line_{x}\to\Line_{d(x)}$, we obtain
  \begin{align*}
  n-1&=\sum_{i=1}^k \sum_{j=1}^{q(i)} (\alpha_j^{i} - 1) =
   \sum_{i=1}^k (n-\abs{A_i}) \quad\text{and}\\
  n-1&=\sum_{i=1}^k \sum_{j=1}^{q(i)} \sum_{J=1}^{r(i)}
   \gcd(\alpha_j^{i},\beta_J^{i})\cdot
   \biggl(\frac{\lcm(\alpha_j^{i},\beta_J^{i})}{\beta_J^{i}}-1\biggr)  \\
    &=\sum_{i=1}^k \sum_{\alpha\in A_i}\sum_{\beta\in B_i}
   (\alpha-\gcd(\alpha,\beta)).
  \end{align*}
  Combined with the analogous expressions for $m-1$, these equations imply
  that $A_i$ and $B_j$ satisfy the hypotheses of Lemma~A below, in which
  we will determine the possibilities for $A_i$ and $B_j$.
  We now determine the corresponding polynomials.  If (C1) holds then some $i$
  has these properties: $P_i$ has a unique preimage in $\Line_{b(x)}$, there is
  a unique $\hat J$ for which $s:=\beta_{\hat J}^i$ is coprime to $n$,
  and further $n\mid \beta_J^i$ for all $J\ne \hat J$.
  By replacing $a$ and $c$ by
  $\ell_1\circ a$ and $\ell_1\circ c$ with $\ell_1$ linear, we may assume
  $P_i=0$.  By replacing $a$ and $b$ by
  $a\circ\ell_2$ and $\ell_2\iter{-1}\circ b$, we may assume $b=0$ is the
  unique root of $a$ (and also that $b=1$ lies above $t=1$), and likewise we
  may assume $d=0$ is the unique root of $c$ having multiplicity $s$.
  Then $a=X^n$ and $c=X^s H(X)^n$ for some $H\in\C[X]$.
  For any $I$, $j$, and $J$, each point of $\Line_x$ lying over
  $Q_j^I$ and $R_J^I$ has ramification index
  $\alpha_j^I/\gcd(\alpha_j^I,\beta_J^I)$ in $\Line_x\to\Line_{d(x)}$.  Since 
  $\alpha_1^i=n$ divides $\beta_J^i$ whenever $Q_J^i\ne 0$, and also
  $\alpha_j^I=1$ if $I\ne i$ (because $n-1=\sum_I (n-\abs{A_I})$
  and $n=\sum_j \alpha_j^I$), it follows that $d=0$ is
  the unique finite branch point of $d$.
  Upon composing $b$ and $d$ on the right with a linear, we may assume $d=X^n$.
  Then $a\circ b = c\circ d$ becomes $b^n = X^{sn} H(X^n)^n$, whence
  $b=\zeta X^s H(X^n)$ with $\zeta^n=1$.  Replacing $H$ by $\zeta H$ puts the
  quadruple $(a,b,c,d)$ in the form (\ref{Ritt2}.2).
  By symmetry, $(c,d,a,b)$ has this form (after composing with linears) if (C2)
  holds.  So assume the $A_i$ and $B_j$ satisfy (C3).
  Then $f=a\circ b$ has just two finite branch points, and every finite
  ramification point has ramification index at most $2$, so Lemma~\ref{ramchar}
  implies that $f$ is dihedral.  Now the result follows from
  Lemma~\ref{chebcheb}.
\end{proof}

\begin{lemA}
  Pick coprime $m,n>1$, and let $A_1,\dots,A_k$ and $B_1,\dots,B_k$
  be multisets of positive integers such that, for each $i$, either $A_i$ or
  $B_i$ (or both) contains an integer greater than $1$.  Suppose further that
  \begin{gather}
  \tag{H1} \label{A1}
    \sum_{\alpha\in A_i} \alpha = n\,\,\, \text{ and }\,\,\,
    \sum_{\beta\in B_j} \beta = m\,\,\, \text{ for each $1\le i\le k$;} \\
  \tag{H2} \label{A2}
    \sum_{i=1}^k (n-\abs{A_i}) = n-1 = \sum_{i=1}^k \sum_{\alpha\in A_i}
    \sum_{\beta\in B_i} (\alpha - \gcd(\alpha,\beta));\, \text{ and} \\
  \tag{H3} \label{A3}
    \sum_{i=1}^k (m-\abs{B_i}) = m-1 =  \sum_{i=1}^k \sum_{\alpha\in A_i}
                    \sum_{\beta\in B_i} (\beta - \gcd(\alpha,\beta)).
  \end{gather}
  Then one of these holds:
  \begin{enumerate}
    \item[(C1)]\label{C1} For some $i$ we have $A_i=\{n\}$, one element of
      $B_i$ is coprime to $n$, and all other elements of $B_i$ are divisible by
      $n$; or
    \item[(C2)]\label{C2} For some $i$ we have $B_i=\{m\}$, one element of
      $A_i$ is coprime to $m$, and all other elements of $A_i$ are divisible by
      $m$; or
    \item[(C3)]\label{C3} $k=2$ and the largest element of
      $A_1\cup A_2\cup B_1\cup B_2$ is $2$.
  \end{enumerate}
\end{lemA}

\begin{proof}
  If $\abs{A_1}=1$ then \eqref{A1} implies $A_1=\{n\}$; thus, by \eqref{A2}, at
  most one element $\hat\beta$ of $B_1$ is not divisible by $n$.  Since
  $n$ is coprime to $m=\sum_{\beta\in B_1} \beta$, it follows that $n$ is
  coprime to $\hat\beta$, so (C1) holds.  Similarly, if $\abs{B_1}=1$ then (C2)
  holds.  Henceforth we assume $\abs{A_i},\abs{B_i}>1$ for each $i$; by
  \eqref{A2}, we have $\abs{A_i}<n$ for at least two values $i$ (so $k>1$), and
  also $\abs{B_j}<m$ for at least two values $j$.  We may assume
  $\abs{A_1},\abs{A_2}<n$.

  Now suppose that, for each $i$ with $1\le i\le k$, we have
  \begin{equation}
  \label{rh}\tag{P1}
    n-\abs{A_i} = \sum_{\alpha\in A_i}\sum_{\beta\in B_i}
     (\alpha-\gcd(\alpha,\beta)).
  \end{equation}
  We first show that $\abs{B_i}\le (m+1)/2$ for $i\in\{1,2\}$.  If
  $1\notin B_i$ then \eqref{A1} implies the stronger inequality
  $\abs{B_i}\le m/2$, so assume $1\in B_i$.  Since $\abs{A_i}<n$, by \eqref{rh}
  and \eqref{A1} we see that $B_i$ contains precisely one copy of $1$, and
  every other element of $B_i$ is divisible by every element of $A_i$.  But
  some element of $A_i$ is at least $2$ (since $\abs{A_i}<n$), so all but one
  element of $B_i$ is at least $2$, whence $\abs{B_i}\le (m+1)/2$ with equality
  just when $B_i=\{1,2,2,\dots,2\}$ and every element of $A_i$ is at most $2$.

  Now $m-1=\sum_{i=1}^k (m-\abs{B_i}) \ge (m-\abs{B_1}) + (m-\abs{B_2}) \ge
  m-1$, so for $i\in\{1,2\}$ we have $\abs{B_i}=(m+1)/2$, whence
  $B_i=\{1,2,2,\dots,2\}$ and every element of $A_i$ is at most $2$.
  Moreover, if $k>2$ then $\abs{B_3}=m$, so $B_3=\{1,1,\dots,1\}$; thus
  \eqref{rh} says that $n-\abs{A_3}=\sum_{\alpha\in A_3}m(\alpha-1)=
  m(n-\abs{A_3})$, so $\abs{A_3}=n$ and $A_3=\{1,1,\dots,1\}$, contradiction.
  This gives (C3), and concludes the proof if (\ref{rh}) holds for every $i$.

  For each $i$ with $1\le i\le k$, and each $\hat\alpha\in A_i$ and
  $\hat\beta\in B_i$, define
  \begin{gather*}
    z(i,\hat\alpha) := 1-\hat\alpha + \sum_{\beta\in B_i}
      (\hat\alpha - \gcd(\hat\alpha,\beta)) \\
    y(i,\hat\beta) := 1-\hat\beta + \sum_{\alpha\in A_i}
      (\hat\beta - \gcd(\alpha,\hat\beta)) \\
    Z(i) := \sum_{\alpha\in A_i} z(i,\alpha) = \abs{A_i}-n +
      \sum_{\alpha\in A_i}\sum_{\beta\in B_i} (\alpha-\gcd(\alpha,\beta)) \\
    Y(i) := \sum_{\beta\in B_i} y(i,\beta) = \abs{B_i} - m +
      \sum_{\alpha\in A_i}\sum_{\beta\in B_i} (\beta-\gcd(\alpha,\beta)).
  \end{gather*}
  Thus $\sum_{i=1}^k Z(i)=0$, and we have already proved the result if every
  $Z(i)=0$, so we may assume $Z(1)<0$ (because $Z(i)=0$ if $\abs{A_i}=n$).
  Likewise we may~assume $Y(I)<0$ for some $I$.
  We will deduce a contradiction.  We compute
  \begin{align}
    \sum_{i=1}^k \Bigl(mn - \sum_{\alpha\in A_i}\sum_{\beta\in B_i}
      \gcd(\alpha,\beta)\Bigr) &=
      \sum_{i=1}^k \sum_{\alpha\in A_i}\sum_{\beta\in B_i}
      (\alpha\beta - \gcd(\alpha,\beta)) \notag\\
    &= m-1 + \sum_{i=1}^k \sum_{\alpha\in A_i}\sum_{\beta\in B_i}
      (\alpha\beta - \beta) \notag\\
    &= m-1 + \sum_{i=1}^k m(n-\abs{A_i}) \tag{P2} \label{disp} \\
    &= m-1 + m(n-1) =mn-1.\notag 
  \end{align}
  (In the setting of Theorem~\ref{Ritt2}, this is Riemann--Hurwitz
  for $\Line_x\to\Line_t$).
  If $Z(i)<0$ then $1\notin B_i$, so $\abs{B_i}\le m/2$ and thus
  $\sum_{\alpha\in A_i}\sum_{\beta\in B_i} \gcd(\alpha,\beta) \le \abs{B_i}
  \sum_{\alpha\in A_i}\alpha = \abs{B_i} n \le nm/2$; similarly, the same
  conclusion holds if $Y(i)<0$.  But \eqref{disp} implies there is at most one
  $i$ satisfying this conclusion, so $I=1$ and $Y(i),Z(i)\ge 0$ for $i>1$.
  Since $\sum_{i=1}^k Z(i)=0$, we have $Z(i)\le -Z(1)$ for $i>1$, and likewise
  $Y(i)\le -Y(1)$.  Since $Z(1)<0$, also $z(1,\alpha)<0$ for some
  $\alpha\in A_1$.  Thus $\alpha$ is not coprime to any element of $B_1$, and
  $\alpha$ divides all but at most one element of $B_1$, so there exists $D>1$
  dividing both $\alpha$ and every element of $B_1$.  Then
  $\sum_{\beta\in B_1} \beta = m$ is divisible by $D$, so $D$ is coprime to $n$
  and thus some $\alpha'\in A_1$ is not divisible by $D$.  For $\beta\in B_1$
  we have $\gcd(\alpha',\beta) \le \beta/2$, so
  $\sum_{\beta\in B_1} (\beta - \gcd(\alpha',\beta)) \ge m/2$.  Also
  $\abs{B_1}\le m/D\le m/2$, so $m-\abs{B_1}=m/2+\delta$ with $\delta\ge 0$,
  and similarly $n-\abs{A_1}=n/2+\gamma$ with $\gamma\ge 0$.  Thus
  $Y(1) \ge \abs{B_1}-m + \sum_{\beta\in B_1} (\beta-\gcd(\alpha',\beta)) \ge
  -\delta$, so $Y(i)\le \delta$ for any $i>1$.  For any $i>1$ we have
  $n-\abs{A_i}\le n-1-(n/2+\gamma)$ (by \eqref{A2}), so
  $\abs{A_i}\ge n/2+\gamma+1$, whence the number of $1$'s in $A_i$ is at least
  $\sum_{\alpha\in A_i} (2-\alpha) = 2\abs{A_i}-n\ge 2(\gamma+1)$.  Thus
  $Y(i) \ge \abs{B_i}-m +2(\gamma+1)(m-\abs{B_i})$, so
  $\delta \ge (2\gamma+1)(m-\abs{B_i})$.  Since $\abs{B_i}<m$ for some $i>1$,
  we obtain $\delta\ge 2\gamma+1$.  Similarly,
  $\gamma\ge 2\delta+1\ge 4\gamma+3$, which is impossible since $\gamma\ge 0$.
\end{proof}

\begin{Remarkapp}
  The proof of Theorem~\ref{Ritt2} becomes simpler if we assume in addition
  that $a$ and $b$ are indecomposable, or more generally that neither $a$ nor
  $c$ has the form $\ell \circ X^e \circ f$ with $\ell$ linear, $e>1$, and $f$
  not a power of a linear polynomial.  The latter condition is equivalent to
  requiring that, for each $i$, if $\abs{A_i}>1$ then the elements of $A_i$
  have no common factor exceeding $1$; and similarly for $B_i$.  If
  $\abs{A_i}=1$ then the beginning of the proof of the Lemma shows (C1) holds.
  So assume $\abs{A_i}>1$ for every $i$, and similarly $\abs{B_i}>1$.
  Since the elements of $A_i$ have gcd $1$, for $\beta\in B_i$ we have
  $y(i,\beta) \ge 0$, with equality just when $\beta$ is coprime to an element
  of $A_i$ and divides all other elements of $A_i$.  Thus $Y(i) \ge 0$; since
  $\sum_{i=1}^k Y(i)=0$, it follows that $Y(i)=0$ for every $i$, so
  $y(i,\beta)=0$ for every $i$ and $\beta$, whence the above equality condition
  holds.  In particular, if we pick $i$ such that $\abs{B_i}<m$, then $B_i$
  contains an element $\beta>1$, so $\abs{A_i}\le (n+1)/2$.  Since
  $n-1=\sum_i (n-\abs{A_i})$, there are at most two values $i$ for which
  $\abs{B_i}<m$, so there are exactly two and each satisfies
  $\abs{A_i}=(n+1)/2$, whence $A_i=\{1,2,\dots,2\}$ and further the largest
  element of $B_i$ is $2$.  For any other $i$, \eqref{A2} implies
  $\abs{A_i}=n$, so $A_i$ and $B_i$ consist solely of $1$'s; this contradicts
  our hypothesis, so $k=2$ and thus (C3) holds.  This proves Lemma~A, and the
  theorem follows as above.
\end{Remarkapp}

\begin{Remarkapp}
  Our proof of Theorem~\ref{Ritt2} is a simplified and rearranged version
  of Ritt's proof.  Ritt's proof looks rather different, since he worked in
  terms of the monodromy group of the Riemann surface for $f(x)-z$, and gave a
  cumbersome description of elements of this group via their action on
  branches.  This is logically equivalent to what we did
  above, but it was viewed by some as being unduly difficult.  Consequently,
  several authors rewrote Ritt's proof in other languages, usually under the
  simplifying assumption that $a,b,c,d$ are indecomposable.  In this special
  case, Ritt's proof has
  been rewritten in terms of polynomial arithmetic
  (\cite{Levi}, \cite[\S 2 of Ch.~4]{LN} and \cite{Binder}),
  valuation theory \cite{DW}, and group theory \cite{Mueller}.  
  Ritt's proof of the full Theorem~\ref{Ritt2} has been translated into the
  language of polynomial arithmetic \cite[\S 5]{Schinzelold}, as well as into
  a language closer to ours \cite[Thm.~6.1]{BT}.  There is also a
  valuation-theoretic version of Ritt's proof \cite{Tortrat}, including a
  different proof of Lemma~A.  Finally, as in the previous remark, it is
  easier to prove Theorem~\ref{Ritt2} when neither $a-\alpha$ nor $b-\alpha$ is
  a nontrivial power of a nonlinear polynomial for any $\alpha\in\C$; one can
  deduce the full result from
  this by a different kind of argument \cite{Zannier} (see also
  \cite[\S 1.4]{Schinzel} or \cite[\S 9]{P2}).  A flawed attempt at such an
  approach is \cite[Thm.~2]{Fried-Ritt}.
  Our proof is arranged quite differently from previous ones, and we hope this
  makes it more understandable.
\end{Remarkapp}



\end{document}